\documentclass[12pt]{article}

\usepackage[utf8]{inputenc}
\usepackage{amsmath}
\usepackage{amsfonts}
\usepackage{amsthm}
\usepackage{amssymb}
\usepackage[mathscr]{euscript}
\usepackage{enumitem}

\usepackage{graphicx}
\usepackage{hyperref}
\usepackage{xcolor}
\hypersetup{
  colorlinks,
  linkcolor={red!50!black},
  citecolor={blue!50!black},
  urlcolor={blue!80!black}
}
\usepackage{subcaption}
\usepackage{tikz}

\usepackage[colors]{optsys}             % optsys-cermics set of common macros
\usepackage{comment}

\includecomment{notinthesis}

\ifdefined\Ia\else \def\Ia#1{{\stackops{\cdot}{\mathrm{I}}{-0.1ex}}\np{#1}}\fi
\ifdefined\Ib\else \def\Ib#1{{\stackops{\cdot}{\mathrm{I}}{-2.5ex}}\np{#1}}\fi

\newcommand{\stablpmu}{\Lpmu[\oplus]{1}{\barRR}}

\newcommand{\Posetorder}{\preceq}
\newcommand{\Posetname}{X}
\newcommand{\poset}{\lowercase{x}}
\newcommand{\POSET}{\espace{\Posetname}}
\newcommand{\Poset}{{\Posetname}}
\newcommand{\POSETEXT}{\np{\POSET,{\Posetorder}}}

\newcommand{\Posetorderfirst}{\preceq_{\POSET}}

\newcommand{\Posetnamesecond}{Y}

\newcommand{\Posetordersecond}{\preceq_{\espace{\Posetnamesecond}}}

\newcommand{\OuterClosureBasic}{\mathop{\uparrow}}

\newcommand{\OuterClosureLone}[1]{\raisebox{-0.4mm}{${\substack{\mathop{\uparrow} \\ \text{\tiny ${(1)}$}}}$}\,{#1}}

\ifdefined\Lpmu\else\newcommand{\Lpmu}[3][]{L_{#1}^{#2}\np{\Omega,\tribu{F},\mu;#3}}\fi
\newcommand{\winf}{\bwedge}% \mathop{\wedge}} %w(edge)inf
\newcommand{\ChoquetInt}{\int^{\mathcal{C}}}

\newcommand{\controlfunc}{\mathfrak{u}} % Fonction u:omega -> U des contrôles
\newcommand{\Controlfunc}{\mathcal{U}}

\newcommand{\appendixref}[1]{Appendix~\ref{#1}}

\title{Minimization Interchange Theorem on Posets}
\author{Jean-Philippe \textsc{Chancelier} \and Michel \textsc{De Lara} \and Beno\^{i}t \textsc{Tran}}
\author{Jean-Philippe Chancelier\thanks{CERMICS, \'{E}cole des Ponts ParisTech, France}
  \and Michel De Lara\footnotemark[1]
  \and Benoît Tran\footnotemark[1]}

\begin{document}
\maketitle

\begin{abstract}
  Interchange theorems between minimization and integration are useful
  in optimization, especially in optimal control and in stochastic optimization.
  In this article, we establish a generalized minimization interchange theorem,
  where integration is replaced by a monotone mapping between posets (partially
  ordered sets). As an application, we recover, and slightly extend, classical results from the literature,
  and we tackle the case of the Choquet integral.
  Our result provides insight on the mechanisms behind existing interchange results.
\end{abstract}

% \mdl{il y a des OCLC dans la biblio}
% \benoit{To do : repasser sur le document et ajouter les macros (directed sup, poset ...) là où il le faut.}

% post abstract

\section{Introduction}

The question of interchanging integration and minimization is an important issue
in stochastic optimization (where integration corresponds to mathematical
expectation).
Let  $\barRR = \RR \cup \na{+\infty} \cup \na{-\infty}$.
Loosely stated, given a measured space $\bp{\Omega, \mathcal{F}, \mu}$
and a subset $X \subset \barRR^{\Omega}$ of functions,
an interchange property has the form
\begin{equation}
  \label{eq:bs1}
  \inf_{x\in X} \int_{\Omega} x \dd\mu= \int_{\Omega} \inf_{x\in X} x \dd\mu
  \eqfinp
\end{equation}
In Equation~\eqref{eq:bs1}, one needs to clarify in which sense
the integral $\int$, the infima $\inf_{X\in X} x$ and
$\inf_{x\in X} \int x \dd\mu$ are defined.
Mathematical frameworks and conditions to obtain Equation~\eqref{eq:bs1} can be
found in \cite{Bo.Gi2001, El.Ta2013, Gi2009, Ka.Mc2016, Ro.We2009, Sh.De.Ru2009}.
We detail the contributions of the two references~\cite{Gi2009} and
\cite{Ro.We2009}.
We refer the reader to \appendixref{sec:outer_int} for recalls and notation
regarding extended Lebesgue and outer integrals.

When the subset~$X$, over which minimization is performed, is a subset of
$\Lpmu{1}{\RR}$ and when the integral~$\int$ is the usual Lebesgue integral,
Giner obtained in \cite{Gi2009} a necessary and sufficient condition
for~\eqref{eq:bs1} as follows.
In this case, the space~$\Lpmu{1}{\RR}$ is endowed with the usual
$\mu$-pointwise order, and the infimum is $\inf_{x\in X} x = \essinf_{x\in X}x$,
which is well-defined %for any nonempty subset$U$ of $\inf_{u\in U} u$ by, for example
by \cite[Proposition II.4.1]{Ne1970}.
Given a subset $X \subset \Lpmu{1}{\RR}$ of functions, Giner establishes that
Equation~\eqref{eq:bs1} holds true if and only if, for every finite
family $x_1, \ldots, x_n$ in $X$, we have
\begin{equation}
  \inf_{x\in X} \int_{\Omega} \np{x - \inf_{1\leq i \leq n} x_i }\dd{\mu} \leq 0
  \eqfinp
  \label{eq:Giner}
\end{equation}
However, checking the above condition is not an easy task, as it depends jointly
on the integral~$\int$ and on the subset~$X$. Moreover,
one may wonder if Equation~\eqref{eq:bs1} still holds true for more general
subsets~$X$, containing functions which are integrable in a weaker sense than Lebesgue integrable.

When a subset $X \subset \Lpmu{0}{\barRR}$ of measurable functions is the image of a set~$U$ by a
mapping $f : \Lpmu{0}{\barRR} \to \Lpmu{0}{\barRR}$, \emph{i.e.} $X = f(U)$, a
celebrated theorem of Rockafellar and Wets (\cite[Theorem 14.60]{Ro.We2009})
gives a condition on the mapping~$f$ and a condition on the set~$U$ so that Equation~\eqref{eq:bs1}
holds true.
In this case, we deal with minimization over subsets~$X$ of~$\Lpmu{0}{\RR}$ and
interchange with the outer integral, a generalization of the Lebesgue integral
to $\Lpmu{0}{\barRR}$.

Our contribution is to provide a minimization interchange theorem
where integration is replaced by a monotone mapping
$\Phi : \POSET \to \YY$
between \emph{posets} (partially ordered sets) $\POSET$ and~$\YY$.
More precisely, we provide an abstract interchange theorem of the form
\begin{equation}
  \label{eq:bs2}
  \winf_{\poset \in \Poset} \Phi\np{\poset} = \Phi\bp{\winf_{\poset\in \Poset} \poset}
  \eqfinp
\end{equation}
Several works studied the abstract interchange of Equation~\eqref{eq:bs2} with $\Phi$ not being the integral, for instance \cite{Ak.Fo2018} when $\Phi$ is an $\alpha$-subhomogeneous operator or \cite{Sh2017} when $\Phi$ is a risk measure.
Once assumed conditions on the mapping $\Phi : \POSET \to \YY$ and
structural properties of the sets~$\POSET$ and~$\YY$,
we provide a necessary and sufficient condition so that Equation~\eqref{eq:bs2} holds
true.
Our search for minimal assumptions led us to assume that the sets~$\POSET$
and~$\YY$ are equipped with partial orders, and that the subset $X\subset
\POSET$
--- on which the interchange property is to be checked ---
is included in a complete inf-semilattice to obtain the existence of greatest
lower bound.
Our result is in the lineage of Giner's condition~\eqref{eq:Giner}, as our necessary and sufficient
condition involves both the mapping~$\Phi$ and the set~$\Poset$.

The article is organized as follows.
Sect.~\ref{sec:MIT} is devoted to a minimization interchange theorem on posets.
Sect.~\ref{sec:MIT_int} tackles the question
of interchange between minimization and
different integrations (extended Lebesgue, outer and Choquet integrals), as well
as order preserving functionals,
by specifying the results of Section~\ref{sec:MIT}.
We provide background on extended Lebesgue and outer integrals in
\appendixref{sec:outer_int}.

We hope that our abstract interchange theorem --- together with its application to
different integrals and order preserving functionals ---
provides insight into how one may obtain
interchange between minimization and integration,
or even go beyond the integral case,
like with risk measures in stochastic optimization.
% Of course, although the above discussion concerns interchanging with
% minimization, one can derive results about maximization by considering the~$-$
% operation when possible.

\section{Minimization interchange theorem on posets}
\label{sec:MIT}

In \S\ref{sec:MIT-main}, we present our main result, namely
Theorem~\ref{th:Minimization_IT}, which provides an abstract interchange result in the
form of Equation~\eqref{eq:bs2} for a mapping $\Phi : \POSET \to \YY$
(generalization of the integral) between specific posets.
For this purpose, we define the notion of
$\Phi$-inf-directed subset $\Poset \subset \POSET$,
as it is instrumental to obtain necessary and
sufficient conditions for such an abstract interchange to hold true.
Then, in~\S\ref{sec:sufficient_cond}, we show how the well-known
notion of inf-directed subset is both sufficient
and easier to check for the subset $\Poset \subset \POSET$ to be
$\Phi$-inf-directed.

\subsection{Main result}
\label{sec:MIT-main}

Before stating Theorem~\ref{th:Minimization_IT}, we provide background on posets
and lattices, as well as two new definitions.
\medskip

We say that $\POSETEXT$ is a \emph{poset} when $\POSET$ is a set and
$\Posetorder$ is a partial order on~$\POSET$, that is, a reflexive, antisymmetric and
transitive binary relation.
Examples include \( \RR \) or \( \barRR \) equipped with the classic
order~$\leq$, or mappings with values in a poset and equipped with the
componentwise order.

Consider a poset~$\POSETEXT$ and a subset $\Poset \subset \POSET$.
% Any $\poset' \in \POSET$ such that, for all $\poset \in \Poset$, we have that
% $\poset \Posetorder \poset'$ is called an \emph{upper bound}
% of the set~$\Poset$.
Any $\poset' \in \POSET$ such that, for all $\poset \in \Poset$, we have that
$\poset' \Posetorder \poset$ is called a \emph{lower bound}
of the set~$\Poset$.
If a lower bound $\poset' \in \POSET$ of the set~$\Poset$ is such that
$\poset'' \Posetorder \poset'$, for any other lower bound $\poset'' \in \POSET$
of the set~$\Poset$, then  $\poset'$ is unique and is called
the \emph{greatest lower bound} of the set~$\Poset$.
In that case, it is denoted by $\inf \Poset$ %or $\vee\np{\Poset}$
or, more explicitly, by $\winf_{{\poset} \in \Poset} \poset$.
We say that a poset $\POSETEXT$ is an \emph{inf-semilattice},
if every nonempty finite subset %$\Poset$
of~$\POSET$
has a greatest lower bound.
We say that a poset $\POSETEXT$ is a \emph{complete inf-semilattice},
if every nonempty subset %$\Poset$
of~$\POSET$
has a greatest lower bound.

% We say that the poset $\np{\POSET, \Posetorderfirst}$ is a \emph{complete lattice} if every nonempty subset has both a greatest lower bound and a lowest upper bound.
% %
% We say that a subset $\Poset \subset \POSET$ has the \emph{countable inf property}
% if $\winf_{\poset \in \Poset} x$ exists  in~$\POSET$
% and if there exists a sequence $\sequence{x_n}{n\in \NN}$ in~$\Poset$ such that
% \( \winf_{n \in \NN} \poset_n = \winf_{\poset \in \Poset} x \).
% \medskip

Now, we introduce the notion of \emph{$\Phi$-inf-directed} subset.

\begin{definition}
  \label{def:phi-inf-directed}
  Let $\np{\POSET,\Posetorderfirst}$ be an inf-semilattice and
  $\np{\YY,\Posetordersecond}$ be a complete inf-semilattice
  and $\Phi: \POSET \to \YY$ be a mapping.
  Let $\Poset \subset \POSET$ be a subset of~$\POSET$.
  We say that the subset \emph{$\Poset$ is $\Phi$-inf-directed} if,
  for every finite subset $\tilde{\Poset} \subset \Poset$, we have that
  \begin{equation}
    \label{eq:Minimization_IT-phidirected}
    \winf_{x\in \Poset} \Phi(x) \Posetordersecond \Phi\np{\winf_{x\in  \tilde{\Poset}}x}
    \eqfinp
  \end{equation}
\end{definition}

With this definition, we can now state our main theorem.

\begin{theorem}[Minimization Interchange Theorem]
  \label{th:Minimization_IT}
  Let $\np{\POSET,\Posetorderfirst}$ be a poset and
  $\np{\YY,\Posetordersecond}$ be a complete inf-semilattice.
  Let $\Poset \subset \POSET$ be a subset of~$\POSET$,
  \( \widetilde\POSET \subset \POSET \) be an inf-semilattice such that
  \( \Poset \subset \widetilde\POSET \subset \POSET \),
  and $\Phi : \POSET \to  \YY$ be a mapping
  such that
  \begin{enumerate}[label=$(T_{\arabic*})$, ref=$T_{\arabic*}$]
    % \begin{enumerate}
  \item
    \label{it:Minimization_IT-op}
    the mapping~$\Phi$ is order preserving, \emph{i.e.} for every $x,x'\in \POSET$,
    \begin{equation}
      x \Posetorderfirst x' \Rightarrow \Phi(x) \Posetordersecond \Phi(x')
      \eqfinv
    \end{equation}
  \item
    \label{it:Minimization_IT-liminf-one-seq}
    The element $\winf_{x\in \Poset} x$ exists in the poset~$\POSET$
    and there exists a sequence $\sequence{x_n}{n\in \NN}$ in $\Poset$ such that
    \begin{enumerate}
    \item
      $\winf_{n\in \NN} \poset_n$ exists in $\POSET$ and
      \begin{equation}
        \winf_{n \in \NN} x_n = \winf_{x\in \Poset} x
        \eqfinv
        \label{eq:Minimization_IT-liminf-one-seq}
      \end{equation}
    \item
      the nonincreasing sequence $\sequence{x'_n}{n\in \NN}$ in $\widetilde\POSET$ defined by
      $x'_n = \winf_{k\leq n} x_k$, for all $n\in \NN$,
      satisfies the following inequality
      \begin{equation}
        \label{eq:FatouLikeAssumptionForOnesequence_inf}
        \winf_{n\in \NN} \Phi\np{x'_n }
        \Posetordersecond \Phi\np{\winf_{n\in\NN} x_n}=
        \Phi\np{\winf_{x\in \Poset} x}
        % \Phi\np{\winf_{n\in \NN} x'_n}
        \eqfinp
      \end{equation}
    \end{enumerate}
  \end{enumerate}
  Then, we have the \emph{interchange formula}
  \begin{equation}
    \label{eq:interchange_formula_inf}
    \winf_{x \in \Poset} \Phi(x) = \Phi\bp{\winf_{x\in \Poset} x}
  \end{equation}
  if and only if the subset~$\Poset$ is $\Phi$-inf-directed
  (as in  Definition~\ref{def:phi-inf-directed}).
\end{theorem}

\begin{proof}
  Let $\Phi : \POSET \to \YY$ and $\Poset \subset \POSET$ be given satisfying
  the two assumptions~\eqref{it:Minimization_IT-op} %, \eqref{it:Minimization_IT-minseq},
  and \eqref{it:Minimization_IT-liminf-one-seq}.
  \medskip

  \noindent $\bullet$
  We assume that the subset~$\Poset$ is $\Phi$-inf-directed
  and we prove the interchange formula~\eqref{eq:interchange_formula_inf} by means of
  two inequalities.

  First, using the fact that the mapping~$\Phi$ is order preserving, we have that
  \begin{equation*}
    \Phi\bp{ \mathop{\wedge}_{x\in \Poset} x} \Posetordersecond \Phi\np{ x'}
    \eqsepv \forall x' \in \Poset
    \eqfinv
  \end{equation*}
  where \( \mathop{\wedge}_{x\in \Poset} x \) is well-defined as an element of~$\POSET$ by Assumption~\eqref{it:Minimization_IT-liminf-one-seq}
  Thus, by the assumption that $\np{\YY,\Posetordersecond}$ is a complete inf-semilattice,
  % the greatest lower bound property of~$\YY$,
  we obtain that
  \[
    \Phi\bp{\mathop{\wedge}_{x\in \Poset} x} \Posetordersecond \mathop{\wedge}_{x\in \Poset} \Phi(x)
    \eqfinp
  \]
  Second, we prove the reverse inequality $\mathop{\wedge}_{x\in \Poset} \Phi(x) \Posetordersecond \Phi\bp{\mathop{\wedge}_{x\in \Poset} x}$.
  Using Assumption~\eqref{it:Minimization_IT-liminf-one-seq}, there exists a sequence
  $\sequence{x_n}{n\in \NN}$, whose terms are in~$\Poset$ (hence in the inf-semilattice~\( \widetilde\POSET \)),
  and such that $\winf_{n\in \NN} x_n =\winf_{x\in \Poset} x$ by~\eqref{eq:Minimization_IT-liminf-one-seq}.
  Now, we define a new sequence $\sequence{x'_n}{n\in \NN}$
  by $x'_n = \winf_{k\leq n} x_k$, for all $n\in \NN$.
  So defined, $x'_n$ does not necessarily belong to the subset $\Poset$,
  but belongs to the inf-semilattice~$\widetilde\POSET$ which contains~$\Poset$.
  % Notice that
  % \begin{equation}
  %   \label{eq:Minimization_IT-0}
  %   \winf_{x\in \Poset} x \Posetorderfirst x'_n \Posetorderfirst x_n
  %   \eqsepv \forall n\in \NN
  % \end{equation}
  % as readily follows from $\winf_{n\in \NN} x_n =\winf_{x\in \Poset} x$
  % and from $x'_n = \winf_{k\leq n} x_k$.
  % We deduce that $\winf_{n\in \NN} x'_n =\winf_{x\in \Poset} x$.
  % Indeed, $\winf_{x\in \Poset} x$ is a minorant of the set \( \nset{x'_n}{n\in\NN} \)
  % and, if there were a strictly greater minorant, it would be,
  % by~\eqref{eq:Minimization_IT-0},
  % a strictly greater minorant also for the set \( \nset{x_n}{n\in\NN} \),
  % which is not the case as $\winf_{n\in \NN} x_n =\winf_{x\in \Poset} x$
  % by assumption.
  % %
  % As a consequence, the sequence $\sequence{x'_n}{n\in \NN}$
  % (whose terms do not necessarily belong to the subset~$\Poset$,
  % whereas those of the sequence $\sequence{x_n}{n\in \NN}$ do)
  % is nonincreasing and satisfies the equalities
  % \begin{equation}
  %   \label{eq:Minimization_IT-1}
  %   \winf_{n\in \NN} x'_n = \winf_{n\in \NN} x_n = \winf_{x \in \Poset} x
  %   \eqfinp
  % \end{equation}
  Then, we get
  \begin{align*}
    \winf_{x\in \Poset} \Phi\np{x}
    & \Posetordersecond \Phi \np{\winf_{k\leq n} x_k}
      \intertext{by~\eqref{eq:Minimization_IT-phidirected} as the subset~$\Poset$ is $\Phi$-inf-directed, by assumption,
      and as the set $\nset{x_k}{k\leq n}\subset \Poset$ is finite}
    & = %\Posetordersecond
      \Phi \np{x'_n}
      \intertext{by definition of $x'_n=\winf_{k\leq n} x_k$, so that we deduce }
      \winf_{x\in \Poset} \Phi\np{x} &
                                       \Posetordersecond \winf_{n\in \NN} \Phi \np{x'_n}
                                       \tag{as $\np{\YY,\Posetordersecond}$ is a complete inf-semilattice by assumption}
    \\
    & \Posetordersecond  \Phi\np{\winf_{n\in \NN} x_n}=  \Phi\bp{\winf_{x\in \Poset} x}
  \end{align*}
  by~\eqref{eq:FatouLikeAssumptionForOnesequence_inf}
  Assumption~\eqref{it:Minimization_IT-liminf-one-seq} is satisfied.
  \medskip

  \noindent $\bullet$
  Conversely, we assume that the interchange formula~\eqref{eq:interchange_formula_inf}
  holds true for the subset $\Poset \subset \POSET$, and we show that $\Poset$ is
  $\Phi$-inf-directed.

  For this purpose, we consider a finite subset $\tilde{\Poset} \subset \Poset$,
  and we get
  \begin{align*}
    \winf_{x\in \Poset} \Phi\np{x}
    & =
      \Phi \np{\winf_{x\in \Poset} x}
      \tag{by the interchange formula~\eqref{eq:interchange_formula_inf}}
    \\
    & \Posetordersecond
      \Phi \bp{\winf_{x\in \tilde{\Poset}} x}
  \end{align*}
  since the mapping~$\Phi$ is order preserving and
  $\winf_{x\in \Poset} x \Posetorderfirst \winf_{x\in \tilde{\Poset}} x$.
  \medskip

  This concludes the proof.
\end{proof}

To state a corollary, we introduce the notion of
\emph{sequentially-inf continuity} of a mapping. The name
is suggested by the notion of sequentially order continuity (denoted as
``continuit\'{e} monotone s\'{e}quentielle'' in~\cite[p. 37]{Ne1970}).

\begin{definition}
  \label{def:sequentially-inf-continuous}
  Let $\np{\POSET,\Posetorderfirst}$ be an inf-semilattice and $\np{\YY,\Posetordersecond}$ be a complete inf-semilattice, $\Phi: \POSET \to \YY$ be a mapping and $\underline{x}$ be a given point of $\POSET$.
  We say that the mapping~$\Phi$ is \emph{sequentially-inf continuous at
    $\underline{x}$} when the following property holds true:
  for any nonincreasing sequence $\sequence{x_n}{n\in \NN}$ in~$\POSET$
  such that $\winf_{n\in \NN} x_n $ exists (in~$\POSET$)
  and such that $\winf_{n\in \NN} x_n = \underline{x}$, we have that
  \begin{equation}
    \label{eq:FatouLikeAssumption_inf}
    \winf_{n\in \NN} \Phi\np{x_n } \Posetordersecond  \Phi\np{\winf_{n\in \NN} x_n}=\Phi\np{\underline{x}}
    \eqfinp
  \end{equation}
  Moreover, we say that the mapping~$\Phi$ is \emph{sequentially-inf continuous on the inf-semilattice $\POSET$} if it is sequentially-inf continuous at every $x\in \POSET$.
\end{definition}

This definition is demanding as the inequality
in~\eqref{eq:FatouLikeAssumption_inf} is the reverse of the inequality
obtained when the mapping~$\Phi$ is nondecreasing,
or the inequality given by the Fatou Lemma  when the mapping~$\Phi$ is the
Lebesgue integral,
or the inequality required if the mapping~$\Phi$ is lower semi continuous.
  
\begin{corollary}[Minimization Interchange Corollary]
  \label{cor:Minimization_IT}
  Let $\np{\POSET,\Posetorderfirst}$ be a poset and
  $\np{\YY,\Posetordersecond}$ be a complete inf-semilattice.
  Let $\Poset \subset \POSET$ be a subset of~$\POSET$,
  \( \widetilde\POSET \subset \POSET \) be an inf-semilattice such that
  \( \Poset \subset \widetilde\POSET \subset \POSET \),
  and $\Phi : \POSET \to  \YY$ be a mapping
  such that
  \begin{enumerate}[label=$(C_{\arabic*})$, ref=$C_{\arabic*}$]
    % \begin{enumerate}
  \item
    \label{it:Minimization_IT-op_cor}
    the mapping~$\Phi$ is order preserving, \emph{i.e.} for every $x,x'\in \POSET$,
    \begin{equation}
      x \Posetorderfirst x' \Rightarrow \Phi(x) \Posetordersecond \Phi(x')
      \eqfinv
    \end{equation}
  \item
    \label{it:Minimization_IT-minseq}
    the subset~$\Poset$ has the countable inf property,
    \emph{i.e.} $\winf_{x\in \Poset} x$ exists in the poset~$\POSET$
    and there exists a sequence $\sequence{x_n}{n\in \NN}$ in $\Poset$ such that $\winf_{n\in \NN} \poset_n$ exists in $\POSET$ and
    \begin{equation}
      \winf_{n \in \NN} x_n = \winf_{x\in \Poset} x
      \eqfinv
    \end{equation}

  \item
    \label{it:Minimization_IT-liminf-at-x}
    the mapping~$\Phi$ is {sequentially-inf continuous} at $\underline{x} = \winf_{x\in \Poset} x$,
    when restricted to the inf-semilattice~$\widetilde\POSET$
    (see Definition~\ref{def:sequentially-inf-continuous}).
  \end{enumerate}
  Then, we have the \emph{interchange formula}
  \begin{equation}
    \label{eq:interchange_formula_inf_cor}
    \winf_{x \in \Poset} \Phi(x) = \Phi\bp{\winf_{x\in \Poset} x}
  \end{equation}
  if and only if the subset~$\Poset$ is $\Phi$-inf-directed
  (as in  Definition~\ref{def:phi-inf-directed}).
\end{corollary}

\begin{proof}

  % Before turning to the proof, it is worth noting that
  % Assumption~\eqref{it:Minimization_IT-liminf-one-seq} can be replaced by one of the following
  % stronger assumptions:
  % \begin{enumerate}[label=$(A'_{\arabic*})$, ref=$A'_{\arabic*}$,start=3,topsep=\parsep]
  % \item \label{it:Minimization_IT-liminf-at-x}
  %   The mapping~$\Phi$ is {sequentially-inf continuous} at $\underline{x} = \winf_{x\in \Poset} x$,
  %   when restricted to the inf-semilattice~$\widetilde\POSET$
  %   (see Definition~\ref{def:sequentially-inf-continuous}),
  % \end{enumerate}
  % and
  % \begin{enumerate}[label=$(A''_{\arabic*})$, ref=$A''_{\arabic*}$,start=3,topsep=\parsep]
  % \item \label{it:Minimization_IT-liminf-global}
  %   The mapping~$\Phi$ is {sequentially-inf continuous}
  %   when restricted to the inf-semilattice~$\widetilde\POSET$
  %   (see Definition~\ref{def:sequentially-inf-continuous}),
  % \end{enumerate}
  % as we easily establish the implications \eqref{it:Minimization_IT-liminf-global}$\implies$\eqref{it:Minimization_IT-liminf-at-x}
  % and \eqref{it:Minimization_IT-minseq}+\eqref{it:Minimization_IT-liminf-at-x}
  % $\implies$\eqref{it:Minimization_IT-liminf-one-seq}.
  %

  Let $\Phi : \POSET \to \YY$ and $\Poset \subset \POSET$ be given satisfying
  the three assumptions~\eqref{it:Minimization_IT-op_cor}, \eqref{it:Minimization_IT-minseq},
  and~\eqref{it:Minimization_IT-liminf-at-x} of
  Corollary~\ref{cor:Minimization_IT}.
  We are going to show that the two assumptions~\eqref{it:Minimization_IT-op}
  and \eqref{it:Minimization_IT-liminf-one-seq} of
  Theorem~\ref{th:Minimization_IT} are satisfied.
  There is nothing to show for assumptions~\eqref{it:Minimization_IT-op_cor}
  and~\eqref{it:Minimization_IT-op} that coincide.
  \medskip

  Then, as \eqref{it:Minimization_IT-minseq}
  and \eqref{it:Minimization_IT-liminf-at-x} hold true,
  from the sequence $\sequence{x_n}{n\in \NN}$ given
  by~\eqref{it:Minimization_IT-minseq}, we build the nonincreasing sequence
  $\sequence{x'_n}{n\in \NN}$ given by $x'_n = \winf_{k\leq n} x_k$,
  for all $n\in \NN$.
  As $\winf_{n\in \NN} x_n =\winf_{x\in \Poset} x$
  by~\eqref{it:Minimization_IT-minseq},
  it readily follows that
  \( \winf_{x\in \Poset} x \Posetorderfirst x'_n \Posetorderfirst x_n\),
  for all $n\in \NN$.
  % \begin{equation}
  %   \label{eq:Minimization_IT-0}
  %   \winf_{x\in \Poset} x \Posetorderfirst x'_n \Posetorderfirst x_n
  %   \eqsepv \forall n\in \NN
  % \end{equation}
  % as readily follows from $\winf_{n\in \NN} x_n =\winf_{x\in \Poset} x$
  % and from $x'_n = \winf_{k\leq n} x_k$.
  We deduce that $\winf_{n\in \NN} x'_n =\winf_{x\in \Poset} x$.
  Indeed, $\winf_{x\in \Poset} x$ is a minorant of the set \( \nset{x'_n}{n\in\NN} \)
  and, if there were a strictly greater minorant, it would be
  % ,  by~\eqref{eq:Minimization_IT-0},
  a strictly greater minorant also for the set \( \nset{x_n}{n\in\NN} \),
  which is not the case as $\winf_{n\in \NN} x_n =\winf_{x\in \Poset} x$
  by assumption.
  As a consequence, the sequence $\sequence{x'_n}{n\in \NN}$
  (whose terms do not necessarily belong to the subset~$\Poset$,
  whereas those of the sequence $\sequence{x_n}{n\in \NN}$ do)
  is nonincreasing and satisfies the equalities
  \( \winf_{n\in \NN} x'_n = \winf_{n\in \NN} x_n = \winf_{x \in \Poset} x \).
  Thus, we have shown~\eqref{eq:Minimization_IT-liminf-one-seq},
  which represents half of Assumption~\eqref{it:Minimization_IT-liminf-one-seq} of
  Theorem~\ref{th:Minimization_IT}.

  To prove~\eqref{eq:FatouLikeAssumptionForOnesequence_inf},
  the second half of Assumption~\eqref{it:Minimization_IT-liminf-one-seq} of
  Theorem~\ref{th:Minimization_IT}, we simply use
  Definition~\ref{def:sequentially-inf-continuous}.
  Indeed, Equation~\eqref{eq:FatouLikeAssumption_inf} with the nonincreasing
  sequence $\sequence{x'_n}{n\in \NN}$, which satisfies \( \winf_{n\in \NN} x'_n = \winf_{x \in \Poset} x =
  \underline{x} \),
  gives
  \(       \winf_{n\in \NN} \Phi\np{x'_n }
  \Posetordersecond \Phi\np{\winf_{n\in\NN} x'_n} \),
  from which we readily get~\eqref{eq:FatouLikeAssumptionForOnesequence_inf} as we have shown
  that \( \winf_{n\in \NN} x'_n = \winf_{n\in \NN} x_n = \winf_{x \in \Poset} x
  \).
  \medskip

  This ends the proof.
\end{proof}

\subsection{A sufficient condition for $\Phi$-directed sets}
\label{sec:sufficient_cond}

Given an order preserving and sequentially-inf-continuous mapping $\Phi : \POSET
\to \YY$, where the posets $\POSET$, $\YY$ have sufficient structure, the
Minimization Interchange Theorem~\ref{th:Minimization_IT} shows that a subset $X
\in \POSET$ is $\Phi$-inf-directed if, and only if, we have the abstract
interchange formula
\(  \winf_{\poset \in \Poset} \Phi\np{\poset} =
\Phi\bp{\winf_{\poset \in \Poset} x} \).
However, as made apparent in its name, checking if a subset~$X$ is $\Phi$-inf-directed is
a condition that involves both~$X$ and its image by the mapping~$\Phi$. We give a simple sufficient condition on the subset
$X$ only which ensures that $X$ is $\Phi$-inf-directed
for any order preserving mapping~$\Phi$.

Let $\POSETEXT$ be a poset.
An \emph{inf-directed}\footnote{It is also called a \emph{filtered set}
  \cite{Gi.Ho.Ke.La.Mi.Sc2003}.}
set $\Poset \subset \POSET$ is a nonempty set
with the property that, for every $\poset$, $\poset' \in \Poset$, there exists $x''\in \Poset$ such that
$\poset'' \Posetorder \poset$ and $\poset'' \Posetorder \poset'$.

We now prove in Lemma~\ref{lemma:DirectedImpliesIntegrably_Directed} that any
inf-directed subset $\Poset \subset \POSET$ is
$\Phi$-inf-directed for any order preserving mapping~$\Phi$.

\begin{lemma}[Inf-directed implies $\Phi$-inf-directed]
  \label{lemma:DirectedImpliesIntegrably_Directed}
  Let $\np{\POSET,\Posetorderfirst}$ be an inf-semi\-la\-tti\-ce,
  \( \Poset \subset \POSET \) be a subset, and
  $\np{\YY,\Posetordersecond}$ be a complete inf-semilattice.
  If the subset~$\Poset$ is inf-directed then $\Poset$ is $\Phi$-inf-directed
  for any order preserving mapping $\Phi: (\POSET,\Posetorderfirst) \to (\YY,\Posetordersecond)$.
\end{lemma}

\begin{proof}
  Suppose that \( \Poset \subset \POSET \) is an inf-directed subset of
  $\POSETEXT$, and let $\Phi: (\POSET,\Posetorderfirst) \to
  (\YY,\Posetordersecond)$ be an order preserving mapping.
  We prove that the subset $X$ is $\Phi$-inf-directed.

  For this purpose, we consider a finite subset
  $\tilde{\Poset} \subset \Poset$. Then, by repeated application of the inf-directed
  property to the finite number of elements in the subset~$\tilde{\Poset}$,
  we get that there exists $\tilde{x}\in X$ such that
  $\tilde{x}\Posetorderfirst \wedge_{\poset \in \tilde{X}} \poset$. We therefore
  obtain that
  \begin{align}
    \bwedge_{\poset\in \Poset} \Phi\np{\poset}
    & \Posetordersecond \Phi\np{\tilde{x}}
      \tag{as $\tilde{x} \in \Poset$}
    \\
    & \Posetordersecond \Phi\bp{\wedge_{\poset \in \tilde{X}} \poset}
      \tag{as $\Phi$ is order preserving and $\tilde{x}\Posetorderfirst \wedge_{\poset \in \tilde{X}} \poset$}
      \eqfinv
  \end{align}
  which ensures that $X$ is $\Phi$-inf-directed and concludes the proof.
\end{proof}

The converse is false, \emph{i.e.} $\Phi$-inf-directed subsets are not necessarily inf-directed subsets as detailed now in
Example~\ref{converse-false}.

\begin{example}[The converse of Lemma~\ref{lemma:DirectedImpliesIntegrably_Directed} is false]
  \label{converse-false}
  Consider $\Omega = \RR$ equipped with its Borel $\sigma$-algebra $\mathcal{B}\np{\RR}$
  and Lebesgue measure $\lambda$. Define the inf-semilattice
  $\POSET = \Lpmu[\oplus]{1}{\barRR}$ (the set of measurable functions with
  Lebesgue integrable positive part, see \appendixref{sec:outer_int})
  with the $\mu$-pointwise order and the
  mapping $\Phi : \POSET \to \barRR$ being the (extended) Lebesgue integral.
  We claim that the subset  \( \Poset \subset \POSET \),  defined by
  \( \Poset = \bp{-n\mathbf{1}_{(n,n+1)},\; n\in \NN}  \subset
  \Lpmu[\oplus]{1}{\barRR} \),
  is $\Phi$-inf-directed but not inf-directed.

  First, we calculate $\wedge_{\poset \in \Poset} \Phi\np{\poset} = \wedge_{\poset \in \Poset}\int_{\RR}
  \poset\np{y} \; \lambda(\mathrm{d}y) = \wedge_{n\in \NN} (-n) = -\infty$.
  Second, for every finite subset
  $\tilde{\Poset} = \left\{ \poset_{n_1}, \ldots, \poset_{n_k} \right\} \subset
  \Poset$ of functions, we have that
  \(  -k\max_{1\leq i \leq k} n_{i}
  \leq \Phi\np{\wedge_{\poset\in \tilde{\Poset}} \poset} \).
  Thus, we get that
  \[
    \wedge_{\poset \in \Poset} \Phi\np{\poset}=-\infty \leq
    -k \max_{1\leq i \leq k} n_{i}\leq \Phi\np{\wedge_{\poset\in \tilde{\Poset}} \poset} \eqfinv
  \]
  hence the subset~$X$ is $\Phi$-inf-directed.

  Nevertheless, $\Poset$ is not an inf-directed subset of~$\POSETEXT$.
  Indeed, let, for all $k\in \NN$, the function $\psi_k$ be defined
  by $\psi_k =-k\mathbf{1}_{(k,k+1)}$, and let $n$ and $n'$ in $\NN$ be fixed such
  that $n\not=n'$. Assume that there exists $n''\in \NN$ such that
  $\psi_{n''} \le \psi_n \wedge \psi_{n'}$.
  Then, if $\Poset$ were an inf-directed subset of~$\POSETEXT$,
  we should have, using the definition of the functions
  $\na{\psi_{k}}_{k\in \NN}$, that the support of $\psi_{n''}$ should contain the
  set $(n,n+1)\cup(n',n'+1)$. However no function of~ $\Poset$ has for support the union
  of two such intervals of unit length.

  In this case, we can observe that the interchange between
  integration and minimization holds true. Indeed, on the one hand we have shown above that
  $\wedge_{\poset \in \Poset} \Phi\np{\poset} = -\infty$ and, on the other hand, we have that
  \[
    \Phi\np{\wedge_{\poset \in \Poset}x} \leq \int_{0}^{+\infty} \np{1-y}
    \lambda\np{\mathrm{d} y} = -\infty \eqfinv
  \]
  hence that
  $\wedge_{\poset \in \Poset} \Phi\np{\poset}= -\infty = \Phi\np{\wedge_{\poset \in
      \Poset}x}$.
\end{example}

\section{Applications to minimization on functional spaces}
\label{sec:MIT_int}

This section is devoted to applications of the Minimization Interchange
Theorem~\ref{th:Minimization_IT} (and its Corollary~\ref{cor:Minimization_IT})
to the case of interchange between (an
extension of) the Lebesgue integral and minimization for suitable subsets of
measurable functions.

In~\S\ref{Interchange_between_minimization_and_integration}, we treat the case
of interchange between minimization and integration,
and we recover both interchange theorems of Giner and Rockafellar-Wets.
In~\S\ref{Comparison_with_Shapiro}, we recover
an interchange theorem of Shapiro for order preserving functionals.
Lastly, in~\S\ref{sec:choquet}, we study the case of the
Choquet integral.

\subsection{Interchange between minimization and integration}
\label{Interchange_between_minimization_and_integration}

We consider a measured space $\np{\Omega, \tribu{F}, \mu}$.
We refer the reader to \appendixref{sec:outer_int} for material regarding
extended Lebesgue and outer integrals.
In~\S\ref{sec:mainresult_int}, we apply the abstract results of Section~\ref{sec:MIT}
to the case of suitable subsets of measurable functions,
and obtain a new Theorem~\ref{thm:integral_inf}.
In~\S\ref{Comparison_with_Giner}
and in~\S\ref{Comparison_with_Rockafellar_and_Wets}, we recover
the interchange theorems of Giner and Rockafellar-Wets from Theorem~\ref{thm:integral_inf}.

\subsubsection{Main result with integrals}
\label{sec:mainresult_int}

We apply the abstract results of Section~\ref{sec:MIT} to the case of subsets of
$\POSET = \Lpmu[\oplus]{1}{\barRR}$, the set of measurable functions with
Lebesgue integrable positive\footnote{%
  \emph{Mutatis mutandis}, we could as well consider ${\POSET} =
  \Lpmu[\ominus]{1}{\barRR}$, the set of measurable functions with Lebesgue
  integrable negative part and maximization in lieu of minimization.}
part.
We consider the interchange with the mapping $\Phi :
\POSET \to \barRR$ being the extended Lebesgue integral on~$\POSET$.

We state the main result about the interchange between the extended Lebesgue integral
$\int_{\Omega} : \Lpmu[\oplus]{1}{\barRR}  \to \barRR$ and minimization.

\begin{theorem}
  \label{thm:integral_inf}
  Let $X$ be a subset of $\Lpmu[\oplus]{1}{\barRR}$.
  Then, \( \essinf_{x\in X} x \in \Lpmu[\oplus]{1}{\barRR} \)
  and the following equality
  \begin{equation}
    \inf_{x \in X} \int_{\Omega}{x}\dd{\mu}= \int_{\Omega}{\essinf_{x\in X} x} \dd{\mu}
  \end{equation}
  is valid if an only if $X$ is \emph{integrably inf-directed},
  \emph{i.e.} for every finite family $x_1, \ldots, x_n$ in $X$ we have
  \begin{equation}
    % \label{InterchangeFromula_Giner}
    \inf_{x\in X} \int_{\Omega} x \dd{\mu} \leq
    \int_{\Omega} \inf_{1\leq i \leq n} x_i \dd{\mu}
    \eqfinp
  \end{equation}
\end{theorem}

\begin{proof}
  As being \emph{integrably inf-directed} defined here coincides with
  being $\Phi$-inf-directed (see Definition~\ref{def:phi-inf-directed}) when
  $\Phi = \int_{\Omega}$ is the extended Lebesgue integral
  on $\Lpmu[\oplus]{1}{\barRR}$, we will show that the assumptions of
  Corollary~\ref{cor:Minimization_IT} % Theorem~\ref{th:Minimization_IT}
  are fulfilled to obtain Theorem~\ref{thm:integral_inf}
  as a special case.

  The proof is broken into two parts. First, the assumptions of
  Corollary~\ref{cor:Minimization_IT} % Theorem~\ref{th:Minimization_IT}
  are satisfied by Proposition~\ref{semiInt_completeSemiLattice}. Namely, the structural
  assumptions on the domain of $\Phi : \Lpmu[\oplus]{1}{\barRR} \to \barRR$ are satisfied
  (see \S\ref{sec:MIT-main} for recalls on the notions below):
  \begin{itemize}
  \item
    The set $\POSET = \widetilde{\POSET} = \Lpmu{0}{\barRR}$ with the $\mu$-pointwise order is a complete inf-semilattice;
  \item
    Every subset $ \Poset \subset \widetilde{\POSET} = \POSET$ has the countable inf property.
  \end{itemize}
  Moreover, $\YY = \barRR$ with the usual order is a complete inf-semilattice.
  Second, by Proposition~\ref{semiint_seqcontinuous},
  the extended Lebesgue integral $\int_{\Omega} : \Lpmu[\oplus]{1}{\barRR} \to \barRR$ is
  order preserving and
  sequentially-inf continuous.

  This ends the proof.
\end{proof}

Note that, as semi-integrable functions --- that is, measurable functions with
either Lebesgue integrable positive part
or Lebesgue integrable negative part ---
are linked by the relation (see Lemma~\ref{lemma:minus_Lebesgue_integral})
\(
x \in \Lpmu[\ominus]{1}{\barRR} \Leftrightarrow -x \in \Lpmu[\oplus]{1}{\barRR}
\),
one can deduce a symmetric result about the interchange between extended Lebesgue integral and maximization.

% \subsection{Proof of Theorem~\ref{thm:integral_inf}}
% and Theorem~\ref{thm:integral_sup}}
% \label{sec:proof_thm_int}

We check in Proposition~\ref{semiInt_completeSemiLattice} (structural properties
of the spaces of measurable and semi-integrable functions) and
Proposition~\ref{semiint_seqcontinuous} (properties of the outer integral) that
the assumptions of the Minimization Interchange
Corollary~\ref{cor:Minimization_IT} % Theorem~\ref{th:Minimization_IT}
are satisfied.

% \subsection{The posets $\Lpmu{0}{\barRR}$, $\Lpmu[\oplus]{1}{\barRR}$ and $\Lpmu[\ominus]{1}{\barRR}$}

% We show that $\Lpmu[\oplus]{1}{\barRR}$ with the $\mu$-pointwise order is a
% complete semilattice and has the countable inf property.

\begin{proposition}[Structural properties of the space of measurable and semi-integrable functions] \label{semiInt_completeSemiLattice}
  \quad
  The set $\Lpmu{0}{\barRR}$ and its subset~$\Lpmu[\oplus]{1}{\barRR}$, both equipped with the $\mu$-pointwise order,
  are complete inf-semilattice with the countable inf property.
\end{proposition}

\begin{proof} \quad

  % voir aussi Lemme 2.6.1 de meyer nieberg banach lattices
  \noindent $\bullet$ We consider the set $\Lpmu{0}{\barRR}$. First, the fact that
  it is a complete inf-semilattice is a consequence of the existence of the essential essential infimum
  for any family (countable or not) of class of random variables as proved in \cite[Proposition II.4.1]{Ne1970}
  (the proof is for probability measures but it extends easily to $\sigma$-finite measures).
  We rephrase here the existence result of~\cite[Proposition II.4.1]{Ne1970}. For any class family (countable or not)
  $\left\{ x_i\right\}_{i\in I}$ in $\Lpmu{0}{\barRR}$, there exists a unique
  class $\essinf_{i \in I} x_i \in \Lpmu{0}{\barRR}$ which is a \emph{greatest lower bound} of the
  family $\left\{ x_i\right\}_{i\in I}$. That is, for any function
  $\underline{x} \in \Lpmu{0}{\barRR}$, we have
  \begin{equation*}
    \forall i\in I, \underline{x} \leq x_i \Leftrightarrow \underline{x} \leq \essinf_{i\in I} x_i
    \eqfinp
  \end{equation*}
  The fact that there exists a countable subfamily $\left\{ x_{i_n} \right\}_{n\in \NN}$ such that
  \begin{equation*}
    \essinf_{i \in I} x_i = \inf_{n\in \NN} x_{i_n}
  \end{equation*}
  is not stated explicitly in~\cite[Proposition II.4.1]{Ne1970}, but it is stated in the proof as
  an intermediate result to obtain the essential infimum.
  It is immediate that the
  countable subfamily can be chosen as a nonincreasing sequence,
  a property that will be useful right below.
  \medskip

  \noindent $\bullet$
  We consider the set $\Lpmu[\oplus]{1}{\barRR}$ and consider a
  class family (countable or not) $\left\{ u_i\right\}_{i\in I}$ in $\Lpmu[\oplus]{1}{\barRR}$.
  As $\Lpmu[\oplus]{1}{\barRR}$ is a subset of $\Lpmu{0}{\barRR}$ we obtain (using the first part of
  the proof) the existence
  of $\essinf_{i \in I} u_i \in \Lpmu{0}{\barRR}$ and the existence of a nonincreasing countable
  subfamily $\left\{ u_{i_n} \right\}_{n\in \NN}$  such that
  \begin{equation*}
    \essinf_{i \in I} u_i = \inf_{n\in \NN} u_{i_n}
    \eqfinp
  \end{equation*}
  Using the monotone convergence theorem for $\Lpmu[\oplus]{1}{\barRR}$
  (see Proposition~\ref{MCT_in_L1moins} in \appendixref{sec:outer_int}), we obtain that
  $\essinf_{i \in I} u_i \in \Lpmu[\oplus]{1}{\barRR}$, as the infimum of a sequence in
  $\Lpmu[\oplus]{1}{\barRR}$. As a consequence,
  the subset~$\Lpmu[\oplus]{1}{\barRR}$ is a complete inf-semilattice which
  has the countable inf property.

  This ends the proof.
\end{proof}

% \subsection{Generalities on the outer integral}

\begin{proposition}[Properties of the outer and extended Lebesgue integrals]
  \label{semiint_seqcontinuous}
  \quad
  \begin{itemize}
  \item
    The outer integral~\eqref{eq:Le-outer-L0} %{eq:Le-outer-L0-Rbar}
    is an order preserving mapping between the posets
    \( \Lpmu{0}{\barRR} \) and $\barRR$.
  \item
    The extended Lebesgue integral~\eqref{eq:int_L1rondeplus} is both order preserving and
    sequ\-en\-tia\-lly-inf continuous on the inf-semilattice $\Lpmu[\oplus]{1}{\barRR}$.
  \end{itemize}
\end{proposition}

\begin{proof}
  Following Definition~\ref{def:Le-outer-L0} the outer integral is clearly order preserving between $\Lpmu{0}{\barRR}$ and $\barRR$. From Proposition~\ref{outerIsExtendedL}, both outer and extended Lebesgue integrals coincide on $\Lpmu[\oplus]{1}{\barRR} \subset \Lpmu{0}{\barRR}$, thus the extended Lebesgue integral is also order preserving between $\Lpmu[\oplus]{1}{\barRR}$ and $\barRR$. We prove that the extended Lebesgue integral is sequentially-inf continuous on $\Lpmu[\oplus]{1}{\barRR}$ using the extended monotone convergence Theorem~\ref{MCT_in_L1moins}.
  Let $\np{\fonctionun_n}_{n\in \NN}$ be an nonincreasing sequence of functions
  in~$\Lpmu[\oplus]{1}{\barRR}$. We put \( \fonctionun =\winf_{n \in \NN}
  \fonctionun_n \), which belongs to the complete inf-semilattice~$\Lpmu{0}{\barRR}$.
  By Proposition~\ref{MCT_in_L1moins}, we get that $\fonctionun \in
  \Lpmu[\oplus]{1}{\barRR}$ and that
  \( \winf_{n \in \NN} \int\fonctionun_n d\mu  = \int\fonctionun d\mu \)
  by~\eqref{eq:extended_monotone_convergence_theorem}.
  Thus, the outer integral~\eqref{eq:Le-outer-L0} is
  sequentially-inf continuous on the inf-semilattice $\Lpmu[\oplus]{1}{\barRR}$.
\end{proof}

% \subsection{Comparison with the literature}
% \label{sec:comparisons_lit}

\subsubsection{Comparison with Giner~\cite{Gi2009}}
\label{Comparison_with_Giner}

From Theorem~\ref{thm:integral_inf}, we now recover the interchange theorem of
Giner.

\begin{theorem}(\cite[Theorem 4.2]{Gi2009})
  \label{InterchangeGiner} Let $X$ be a subset of $\Lpmu{1}{\RR}$.
  The following equality
  \begin{equation}
    \inf_{x \in X} \int_{\Omega}{x}\dd{\mu}= \int_{\Omega}{\essinf_{x\in X} x} \dd{\mu}\; ,
  \end{equation}
  is valid if an only if $X$ is \emph{integrably inf-directed}, \emph{i.e.} for any finite family $x_1, \ldots, x_n$ in $X$ we have
  \begin{equation}
    \label{InterchangeFromula_Giner}
    \inf_{x\in X} \int_{\Omega} \np{x - \inf_{1\leq i \leq n} x_i }\dd{\mu} \leq  0 \; .
  \end{equation}
\end{theorem}

As $\Lpmu{1}{\RR} \subset \Lpmu[\oplus]{1}{\barRR}$,
the interchange formula in Theorem~\ref{thm:integral_inf} is a slight generalization to $\Lpmu[\oplus]{1}{\barRR}$
of Giner's Theorem~\ref{InterchangeGiner} stated for subsets of $\Lpmu{1}{\RR}$.
This is no surprise, as we are indebted to Giner since
Theorem~\ref{thm:integral_inf} was greatly inspired by Giner's result.

\subsubsection{Comparison with Rockafellar and Wets \cite{Ro.We2009}}
\label{Comparison_with_Rockafellar_and_Wets}

We prove that the Rockafellar-Wets interchange theorem below can be deduced from
Theorem~\ref{thm:integral_inf} combined with~\cite[Theorem 3.1]{Gi2009}.

Let $\np{\Omega, \mathcal{F}, \mu}$ be a measured space with $\mu$ being a $\sigma$-finite measure.
As we work with subsets of measurable functions, the integral used here is the
outer integral \( \int^{*}_{\Omega} \)
(see Definition~\ref{def:Le-outer-L0} in \appendixref{sec:outer_int}).
Following \cite{Ro.We2009}, a subset $\Controlfunc \subset \Lpmu{0}{\RR^d}$ is said to be \emph{Rockafellar-Wets decomposable}
(w.r.t. the $\sigma$-finite measure $\mu$) if
\begin{equation}
  y \mathbf{1}_A + \controlfunc \mathbf{1}_{A^c} \in \Controlfunc
  \eqsepv
  \forall y \in L^{\infty}(A,\tribu{F},\mu ; \RR^d)
  \eqsepv
  \forall A \in \tribu{F} \eqsepv \mu\np{A} < +\infty
  \eqsepv
  \forall \controlfunc \in \Controlfunc
  \eqfinp
\end{equation}
The notion of decomposable subsets is widely used in $L^p$ spaces and we refer
the reader to~\cite{Gi2009} for a survey on various related definitions.

\begin{theorem}(\cite[Theorem~14.60]{Ro.We2009})
  \label{Interchange_Rockafellar-Wets}
  Let $\Controlfunc$ be a subset of $\Lpmu{0}{\RR^d}$ that is Rockafellar-Wets decomposable.
  % \footnote{See \Cref{Rockafellar-Wets_decomposable} with $\RR^d$ instead of $\barRR$}.
  Let $\fonctiondeux: \Omega \times \RR^d \rightarrow \barRR$ be a normal integrand\footnote{See \cite[Definition~14.27]{Ro.We2009}.}. If
  there exists $\bar{\controlfunc}\in \Controlfunc$ such that $\fonctiondeux \bp{\cdot,\bar{u}\np{\cdot}} \in \Lpmu[\oplus]{1}{\RR^d}$, one
  has that
  \begin{equation}
    \label{eq:interchange-rockafellar}
    \inf_{\controlfunc \in \Controlfunc} \int^{*}_{\Omega} \fonctiondeux \bp{\omega, \controlfunc(\omega)} \dd{\mu}(\omega)
    =\int^{*}_{\Omega} \Bp{ \inf _{\control \in \RR^{n}} \fonctiondeux (\omega, \control) }
    \dd{\mu}(\omega)
    \eqfinp
  \end{equation}
  Moreover, as long as this common value is not $-\infty,$ one has that
  \begin{equation*}
    \underline{\controlfunc} \in \argmin_{\controlfunc \in \Controlfunc}
    \int^{*}_{\Omega} \fonctiondeux \bp{\omega, \controlfunc(\omega)} \dd{\mu}(\omega)
    \iff
    \underline{\controlfunc} \in \Controlfunc \mtext{ and }
    \underline{\controlfunc}\np{\cdot} \in
    \argmin_{\control \in \RR^{n}} \fonctiondeux (\cdot, \control) \text{\Pps[\mu]}
    \eqfinp
  \end{equation*}
\end{theorem}

The proof relies on the property that the image by a measurable mapping of a Rockafellar-Wets decomposable subset is an integrably inf-directed subset of $\Lpmu{0}{\barRR}$.

\begin{proof}(Equation~\eqref{eq:interchange-rockafellar} as a consequence of \cite[Theorem 3.1]{Gi2009} and the
  Minimization Interchange Corollary~\ref{cor:Minimization_IT}) % Theorem~\ref{th:Minimization_IT})
  \medskip

  \noindent $\bullet$ We introduce the set
  $X = \bset{\Omega \ni \omega \mapsto \fonctiondeux \np{\omega, \controlfunc(\omega)}}{\controlfunc \in \Controlfunc}$. Using the fact
  that the function~$\fonctiondeux$ is a normal integrand and that $\Controlfunc$ is a subset of
  $\Lpmu{0}{\RR^d}$, we obtain that $X$ is a subset of
  $\Lpmu{0}{\barRR}$ \cite[Theorem 14.37]{Ro.We2009} and we can write
  \begin{align*}
    % \label{eq:rockafellar-transform}
    \inf_{\controlfunc \in \Controlfunc} \int^{*}_{\Omega} \fonctiondeux \bp{\omega, \controlfunc(\omega)} \dd{\mu}(\omega)
    & =
      \bwedge_{x\in X} \int^{*}_{\Omega} x\np{\omega} \dd{\mu}(\omega)
      \eqfinp
      % \\
      % &= \bwedge_{u \in U \cap \Lpmu[\oplus]{1}{\barRR}} \int^{*}_{\Omega} u\np{\omega} \dd{\mu}(\omega)
      % \eqfinp
  \end{align*}
  Now, using the definition~\eqref{eq:Le-outer-L0} of the outer integral, we have that
  \begin{align*}
    \label{eq:rockafellar-transform-two}
    \bwedge_{x\in X} \int^{*}_{\Omega} x\np{\omega} \dd{\mu}(\omega)
    =
    \bwedge_{x\in X} \inf_{\substack{x' \in \Lpmu{1}{\RR} \\ x' \ge x}} \int_{\Omega} x'\np{\omega} \dd{\mu}(\omega)
    \eqfinp
  \end{align*}
  We define the upper set $\OuterClosureLone{X}$ of~$X$ in $\Lpmu{1}{\RR}$ by
  \begin{equation*}
    % \label{eq:outerclosure}
    \OuterClosureLone{X}
    = \bset{ x' \in \Lpmu{1}{\RR}}{ \exists x \in X \text{ s.t. } x \le x' \text{ \Pps[\mu]}}
    \eqfinp
  \end{equation*}
  The set $\OuterClosureLone{X}$ is not empty.
  Indeed, by assumption there exists $\bar{\controlfunc}\in \Controlfunc$
  such that $\overline{x}=\fonctiondeux \bp{\cdot,\bar{\controlfunc}\np{\cdot}} \in
  \Lpmu[\oplus]{1}{\barRR}$.
  As $\overline{x}$ belongs to~$X$, we conclude that $\overline{x}_{+} \in
  \OuterClosureLone{X}$, hence that $\OuterClosureLone{X}$ is not empty.
  Combining the three equations above,
  % Equation~\eqref{eq:rockafellar-transform}-\eqref{eq:outerclosure},
  we readily get that
  \begin{equation*}
    \inf_{\controlfunc \in \Controlfunc} \int^{*}_{\Omega} \fonctiondeux \bp{\omega, \controlfunc(\omega)} \dd{\mu}(\omega)
    =
    \bwedge_{ x' \in \OuterClosureLone{X}} \int_{\Omega} x' \np{\omega} \dd{\mu}(\omega)
    \eqfinp
  \end{equation*}
  % In addition, using the fact that there exists $\bar{u}\in U$
  % such that $\fonctiondeux \bp{\cdot,\bar{u}\np{\cdot}} \in \Lpmu[\oplus]{1}{\barRR}$ we have
  % that $\overline{x}=\fonctiondeux \bp{\cdot,\bar{u}\np{\cdot}}$ belongs to $X$.
  % % belongs to the domain of the
  % % outer integral mapping $\dom \bp{ \int_\Omega^*}$ which is precisely $\Lpmu[\oplus]{1}{\barRR}$. \benoit{We have $\dom \bp{ \int_\Omega^*} \subset \Lpmu[\oplus]{1}{\barRR}$ by Equation~\eqref{eq:outer-as-leb-diff}. But for the converse, the outer integral in $\Lpmu[\oplus]{1}{\barRR}$ can still be $-\infty$. Consider $f = -\infty \mathbf{1}_{[0,1]} \in \Lpmu[\oplus]{1}{\barRR}$ but  $\int_{\Omega}^* f = -\infty$.}
  % Therefore, the set $\OuterClosureLone{X}$ is not empty  as $\overline{x}_{+} \in \OuterClosureLone{X}$.
  \medskip

  \noindent $\bullet$
  By \cite[Proposition 5.4]{Gi2009}, as the function~$\fonctiondeux$ is a normal integrand and thus measurable, the set $\OuterClosureLone{X}$ is integrably inf-directed.
  % \benoit{Changed using Giner's proposition.}

  \noindent $\bullet$
  The last step to obtain~\eqref{eq:interchange-rockafellar} is to prove that
  \[
    \inf _{\control \in \RR^{n}} \fonctiondeux\np{\omega, \control} = \nwedge_{{x' \in \OuterClosureLone{X}}} x'(\omega)
    \eqfinv
  \]
  which is obtained using~\cite[Theorem 3.1]{Gi2009}.
\end{proof}

\subsection{Comparison with Shapiro \cite{Sh2017}}
\label{Comparison_with_Shapiro}

A restricted literature \cite{Pi.Sc2019, Sh2017} considers interchange theorems
not with integration but with more general monotone functionals.
We focus on \cite{Sh2017}, which examines the case of three posets~$\POSET$:
the set of continuous functions over a compact set;
a Euclidean space;
the normed linear space $\Lpmu{p}{\RR}$ equipped with the
$\mu$-pointwise order and its norm $\lVert \cdot \rVert_p$,
where $p\in [1, +\infty]$ and $(\Omega,\tribu{F},\mu)$ is a probability space.
For the sake simplicity, we will only consider the case
$\POSET = \Lpmu{p}{\barRR}$ and leave to the reader the two other cases, as
they can be treated similarly.

We use the Minimization Interchange Theorem~\ref{th:Minimization_IT}
to recover (an extended version of)
the interchange result of \cite{Sh2017}.

\begin{proposition}(extended from \cite[Proposition~2.1]{Sh2017})
  \label{prop:shapiro}
  Let $p\in [1, +\infty)$, $(\Omega,\tribu{F},\mu)$ be a probability space,
  $\bp{\Lpmu{p}{\RR}, \lVert \cdot \rVert_p}$ be the normed linear space of $p$-Lebesgue integrable functions equipped with the
  $\mu$-pointwise order, and $\Phi : \Lpmu{p}{\RR} \to \RR \cup \{ +\infty \}$ be an order
  preserving functional.

  Let $\CONTROL$ be a set and
  $\fonctiondeux: \Omega \times \CONTROL \rightarrow \barRR$ be a function.
  We define the mapping
  $G: {\mathbb{\CONTROL}}^{\Omega} \to {\barRR}^{\Omega}$ by
  $G(\controlfunc): \omega \ni \Omega \mapsto \fonctiondeux\bp{\omega,
    \controlfunc\np{\omega}} \in \barRR$, for all
  $\controlfunc \in {\mathbb{\CONTROL}}^{\Omega}$.
  We denote by $G^{\flat}\in {\barRR}^{\Omega}$ the function $G^{\flat}: \omega \mapsto \inf_{\control
    \in \CONTROL}\fonctiondeux\bp{\omega, \control}$
 and we assume that \(G^{\flat}\in \Lpmu{p}{\RR}\).
  Let $\Controlfunc \subset \mathbb{\Control}^{\Omega}$ be a subset of
  $\mathbb{\Control}^{\Omega}$.

  Suppose that
  \begin{enumerate}[label=$(S_{\arabic*})$,
    ref=$S_{\arabic*}$,start=1,topsep=\parsep]
  \item
    \label{item:subset_of_POSET}
    the image $X = G\np{\Controlfunc}$ of~$\Controlfunc \subset
    \mathbb{\Control}^{\Omega}$ by the mapping~$G: {\mathbb{\CONTROL}}^{\Omega}
    \to {\barRR}^{\Omega}$ is a subset of $\Lpmu{p}{\RR}$,
    that is, \( X = G\np{\Controlfunc} \subset \Lpmu{p}{\RR} \),
  \item
    \label{item:sequence_controlfunc}
    there exists a sequence $\sequence{\controlfunc_n}{n\in \NN}$ in
    $\Controlfunc$ such that
    \begin{enumerate}
    \item
      \label{item:norm-conv-sequence}
      $\lVert G(\controlfunc_n) - G^\flat \rVert_p \underset{n \to
        +\infty}{\longrightarrow} 0$, that is, the sequence $\sequence{G(\controlfunc_n)}{n\in \NN}$ (strongly)
      converges to $G^{\flat}$ in~$\bp{\Lpmu{p}{\RR}, \lVert \cdot \rVert_p}$,
    \item
      \label{item:continuity}
      \( \Phi\bp{G^{\flat}} \geq \underline{\lim}_{n \to+\infty}\Phi\bp{G(\controlfunc_n)} \)
      % the function $\Phi : \Lpmu{p}{\RR} \to \RR \cup \{ +\infty \}$ is (strongly)
      % continuous at~$G^{\flat}$.
      % the function~$G^{\flat}$ belongs to~$\Lpmu{p}{\RR}$ (that is, $G^{\flat} \in \Lpmu{p}{\RR}$), and the mapping
      % $\Phi : \Lpmu{p}{\RR} \to \RR \cup \{ +\infty \}$ is (strongly)
      % continuous at~$G^{\flat}$.
      % \( \Phi\np{G^\flat} \geq \underline{\lim}_{n \to
      % +\infty}\Phi\bp{G(\controlfunc_n)} \)
      (an assumption which holds true when the mapping
      $\Phi : \Lpmu{p}{\RR} \to \RR \cup \{ +\infty \}$ is (strongly)
      continuous at~$G^{\flat}$).
      % \end{enumerate}
    \end{enumerate}
  \end{enumerate}
  Then, we have that
  \begin{equation}
    \label{eq:shapiro0}
    \inf_{x \in X} \Phi\bp{x} = \Phi \bp{G^{\flat}}
    \eqfinp
  \end{equation}
\end{proposition}

\begin{proof}
  With the notation of Sect.~\ref{sec:MIT},
  we set \( \POSET=\Lpmu{p}{\RR} \) which is both a normed linear space and a poset when equipped with the
  $\mu$-pointwise order.

  We prove Equation~\eqref{eq:shapiro0} by showing two equalities
  \begin{equation}
    \label{eq:shapiro1}
    \winf_{x\in X} x = G^{\flat} \text{ and } \inf_{x \in X} \Phi\bp{x} = \Phi\bp{ \winf_{x\in X} x}
    \eqfinv
    % \label{eq:two-step-interchange}
  \end{equation}
  where the right hand side is an interchange formula and the left hand side
  shows that the essential infimum over $X$ is realized by the pointwise
  infimum~$G^{\flat}$, as done in~\cite[Theorem 3.1]{Gi2009}.

  We prove the left hand side equality in Equation~\eqref{eq:shapiro1} as follows.
  First, we prove that the function~$G^\flat$ satisfies $G^\flat \le  \inf_{\controlfunc \in \Controlfunc} G\np{\controlfunc}$.
  By definition of $G^{\flat}: \omega \mapsto
  \inf_{\control \in \CONTROL}\fonctiondeux\bp{\omega, \control}$,
  we have that $G^{\flat}\np{\omega} \le G\np{\controlfunc}\np{\omega}=\fonctiondeux\bp{\omega, \controlfunc\np{\omega}}$ for all $\controlfunc \in \Controlfunc$ and $\omega \in \Omega$.
  Thus, we get that $G^{\flat} \le  G\np{\controlfunc}$ for all $\controlfunc \in \Controlfunc$.

  % and therefore that $G^{\flat} \le \winf_{\controlfunc \in \Controlfunc} G\np{\controlfunc} = \winf_{x \in X} x$.

  Second, we prove that
  $G^{\flat} = \winf_{\controlfunc \in \Controlfunc} G\np{\controlfunc}= \winf_{n
    \in \NN} G(\controlfunc_{n})$, where the
  sequence $\sequence{\controlfunc_n}{n\in \NN}$ is given by
  Assumption~\eqref{item:norm-conv-sequence}.
  Using the just proven property that
  $G^{\flat} \le  G\np{\controlfunc}$ for all $\controlfunc \in \Controlfunc$,
  and the fact that $\controlfunc_n \in\Controlfunc$ for all
  $n \in \NN$, we obtain the following inequalities between functions in
  ${\barRR}^{\Omega}$:
  \begin{equation}
    G^{\flat} \leq
    \winf_{\controlfunc \in \Controlfunc} G\np{\controlfunc}
    \leq \winf_{n \in \NN}G\np{\controlfunc_n}
    \leq G\np{\controlfunc_0}
    \eqfinp
    \label{eq:three_winf_inequalities}
  \end{equation}
  We now show that
  \( G^\flat=  \winf_{n \in \NN}G\np{\controlfunc_n} \)
  where this equality is to be understood as an equality between classes in the
  space~$\Lpmu{p}{\RR}$.
  As $G\np{\Controlfunc}\subset\POSET =\Lpmu{p}{\RR} $  by
  Assumption~\eqref{item:subset_of_POSET}, we get that
  \( G\np{\controlfunc_0} \in \Lpmu{p}{\RR} \), hence so is
  $\winf_{n \in \NN}G\np{\controlfunc_n}$.
  % the function~$G^\flat$ is a lower bound of the sequence
  % $\sequence{G\np{\controlfunc_n}}{n\in \NN}$. Moreover, as $G\np{\Controlfunc}\subset\POSET$  by
  % Assumption~\eqref{item:subset_of_POSET}, the sequence $\sequence{G\np{\controlfunc_n}}{n\in \NN}$ is in $\POSET$. Thus, using the know fact that $\POSET$
  % is $\sigma$-Dedekind complete (every bounded from below nonempty countable subset of $\POSET$  has a greatest lower bound) we obtain that
  % $\winf_{n \in \NN}G\np{\controlfunc_n}$ is in $\POSET$.
  Now, as
  \( 0 \leq \winf_{n' \in \NN} G(\controlfunc_{n'}) -G^{\flat}
  \leq G(\controlfunc_{n}) -G^{\flat} \)
  % $|G^{\flat} - \winf_{n' \in \NN} G(\controlfunc_{n'})| \le |G^{\flat} -G(\controlfunc_{n})|$
  for all $n\in \NN$,
  we obtain the inequality $\norm{G^{\flat} - \winf_{n' \in \NN}
    G(\controlfunc_{n'})}_{p} \le \norm{G^{\flat} - G(\controlfunc_{n})}_{p}$
  between $L^p$-norms.
  As, by Assumption~\eqref{item:norm-conv-sequence}, the sequence
  $\sequence{G(\controlfunc_{n})}{n\in \NN}$ strongly converges to $G^\flat$,
  we deduce that $G^{\flat} = \winf_{n \in \NN} G(\controlfunc_{n})$
  in~$\Lpmu{p}{\RR}$.
  By~\eqref{eq:three_winf_inequalities}, we conclude that
  $G^\flat= \winf_{\controlfunc \in \Controlfunc} G\np{\controlfunc}=  \winf_{n \in \NN}G\np{\controlfunc_n}$
  in~$\Lpmu{p}{\RR}$.
  % %
  % Finally, $\winf_{\controlfunc \in \Controlfunc} G\np{\controlfunc}$ as a function in $\barRR^{\Omega}$ is well defined and satisfies
  % $G^\flat \le \winf_{\controlfunc \in \Controlfunc} G\np{\controlfunc} \le  \winf_{n \in \NN}G\np{\controlfunc_n}$. Hence, we conclude that
  % $\winf_{\controlfunc \in \Controlfunc} G\np{\controlfunc}$ is in $\POSET$ and satisfy
  % $G^\flat= \winf_{\controlfunc \in \Controlfunc} G\np{\controlfunc}=  \winf_{n \in \NN}G\np{\controlfunc_n}$. This concludes the second step.

  Third, we define $x_n =G(\controlfunc_n)$ for all $n\in \NN$,
  with the property that the sequence $\sequence{x_n}{n\in \NN}$ is in~$\POSET$
  by Assumption~\eqref{item:subset_of_POSET}.
  Setting \( \underline{x}=G^{\flat} \), we have obtained the following
  equalities in the poset $\POSET =\Lpmu{p}{\RR} $:
  \begin{equation}
    \underline{x} =G^{\flat} = \winf_{\controlfunc \in \Controlfunc} G\np{\controlfunc}= \winf_{x \in X} x=  \winf_{n \in \NN} x_n
    \eqfinp
    \label{eq:X-countable-inf-property}
  \end{equation}
  We have thus proved the left hand side equality in Equation~\eqref{eq:shapiro1}.
  \medskip

  Now, we prove that the interchange formula in the right hand side of Equation~\eqref{eq:shapiro1}
  is satisfied by checking that the assumptions of Theorem~\ref{th:Minimization_IT} hold true.
  For this purpose, we define the subset
  $\widetilde{\POSET} =\OuterClosureBasic \na{\underline{x}}= \nset{x \in \POSET}{
    x\geq \underline{x}}$ of $\POSET$. Then as a finite infimum of functions in
  $\Lpmu{p}{\RR}$ is also in $\Lpmu{p}{\RR}$, the subset $\widetilde{\POSET}$ is an
  inf-semilattice of the poset $\POSET$.

  As the mapping~$\Phi$ is order preserving,
  Assumption~\eqref{it:Minimization_IT-op}  in Theorem~\ref{th:Minimization_IT}
  holds true.
  Also, we have already proven in~\eqref{eq:X-countable-inf-property}
  that Equation~\eqref{eq:Minimization_IT-liminf-one-seq} holds true,
  which represents half of Assumption~\eqref{it:Minimization_IT-liminf-one-seq} of
  Theorem~\ref{th:Minimization_IT}.
  To prove~\eqref{eq:FatouLikeAssumptionForOnesequence_inf}, the second half of Assumption~\eqref{it:Minimization_IT-liminf-one-seq} of
  Theorem~\ref{th:Minimization_IT}, we consider
  the nondecreasing sequence $\sequence{x'_n}{n\in \NN}$ in the inf-semilattice~$\widetilde{\POSET}$
  defined, for all $n\in \NN$, by $x'_n = \winf_{k\leq n} x_k$.
  As \( \underline{x} \leq x'_n  \leq x_n \) for all $n\in \NN$,
  and as the sequence $\sequence{x_n}{n\in \NN}$ strongly converges to
  $\underline{x}$,
  we readily get that  so does the sequence $\sequence{x'_n}{n\in \NN}$.
  
  Now, we get that
  \begin{align*}
    \Phi\bp{\underline{x}}
    = \Phi\bp{G^{\flat}}
    &\geq
      \underline{\lim}_{n\to+\infty}\Phi\bp{G(\controlfunc_n)}
      \tag{by Assumption~\eqref{item:continuity}}
    \\
    &=
      \underline{\lim}_{n\to+\infty}\Phi\np{x_n}
      \tag{as \( x_n=G(\controlfunc_n) \)}
    \\
    & \geq
      \underline{\lim}_{n\to+\infty}\Phi\np{x'_n}
      \intertext{as the mapping~$\Phi$ is order preserving and as
      \( x_n  \geq x'_n \) }
    &=
      \winf_{n\in \NN} \Phi\np{x'_n}
  \end{align*}
  as the sequence $\sequence{x'_n}{n\in \NN}$ is nondecreasing, hence so is
  the sequence $\sequence{\Phi\np{x'_n}}{n\in \NN}$.
  Thus, we have shown that~\eqref{eq:FatouLikeAssumptionForOnesequence_inf} holds true.

  % by
  % Assumption~\eqref{item:continuity}, we get that
  % \[
  %   \Phi\bp{G^{\flat}} \geq \underline{\lim}_{n
  %   \to+\infty}\Phi\bp{G(\controlfunc_n)}
  % \]

  % the function~$\Phi$ is continuous;
  % as $x'_n  \underset{n \to \infty}\longrightarrow \underline{x}$, we get that
  % \begin{align*}
  %   \winf_{n\in \NN} \Phi\np{x'_n} xxx \leq  \lim_{n \to \infty} \Phi\np{x'_n}
  %   =   \Phi\np{\underline{x}}
  %   \eqfinp
  % \end{align*}
  As a consequence, we have shown that Assumption~\eqref{it:Minimization_IT-liminf-one-seq} in
  Theorem~\ref{th:Minimization_IT} holds true.
  \medskip

  Finally, we prove that the subset~$\Poset$ is $\Phi$-inf-directed.
  For this purpose, we consider a finite subset $\Poset'$ of $\Poset$.
  Recall that $\sequence{x_n}{n\in \NN}$ denotes a sequence in $X$ which strongly converges to $\underline{x} = \winf_{x\in X} x$ and which realizes the infimum over $X$, \emph{i.e.} $\winf_{n\in \NN} x_n = \winf_{x\in X} x$. We successively have
  \begin{align*}
    \Phi\np{ \winf_{x\in \Poset'} x}
    & \geq \Phi\np{ \winf_{x\in \Poset} x}
      \tag{as $\Phi$ is order preserving and $\winf_{x\in \Poset'} x \geq \winf_{x\in \Poset} x$}
    \\
    & = \Phi\np{\underline{x}}
      \eqfinv                                                                                                                        \\
    & \geq
      \underline{\lim}_{n\to+\infty}\Phi\np{x_n}
      \tag{by Assumption~\eqref{item:continuity} as \( x_n=G(\controlfunc_n) \)}
    \\
    & \geq \winf_{n \in \NN} \Phi\np{x_n}
    \\
    & \geq \winf_{x\in X} \Phi\np{x}  \tag{as $x_n\in \Poset$ for all $n\in \NN$}
      \eqfinv
  \end{align*}
  which shows that $\Poset$ is $\Phi$-inf-directed and ends the proof.
\end{proof}

\subsection{Interchange between minimization and Choquet's integral}
\label{sec:choquet}

Let $\np{\Omega, \mathcal{F}}$ be a measurable space. We specialize the
Minimization Interchange Theorem~\ref{th:Minimization_IT} to the poset of
nonnegative measurable functions
\[
  \POSET = \left\{ x : \Omega \to \barRR \mid x \geq 0 \text{ and measurable}  \right\}
\]
with the pointwise order and the Choquet integral $\Phi = \ChoquetInt$ that we
define below. We suggest \cite{Ka2018} and the references therein for properties
of the Choquet integral. One main difference of the Choquet integral compared to
the Lebesgue integral is that it is nonadditive.

A \emph{capacity} $c : \mathcal{F} \to \barRR$ is a function which is order preserving ($\forall F_1, F_2 \in \mathcal{F}, F_1 \subset F_2 \Rightarrow c\np{F_1} \leq c\np{F_2}$) and such that $c\np{\emptyset} = 0$. Given a capacity $c$, the Choquet integral of a nonnegative measurable function $x \in \POSET$ is defined by
\begin{equation}
  \ChoquetInt_{\Omega} x\np{\omega} \dd c \np{\omega} =
  \int_{\RR_+} c\bp{x > t} \dd t \eqfinv
  \label{eq:ChoquetInt}
\end{equation}
where the integral on the right-hand side is the Lebesgue integral of an
nonincreasing function. A capacity $c$ is said to be \emph{continuous from
  above} if, for any nondecreasing sequence $\left\{F_n\right\}_{n\in
  \NN} \subset \mathcal{F}$ of sets such that $F = \cap_{n\in \NN} F_n \in \mathcal{F}$,
we have that $c\np{F_n} \underset{n\to +\infty}{\longrightarrow} F$.
Lastly, we say that a subset $X \subset \POSET$ of functions is \emph{Choquet integrably inf-directed} if
it is integrably inf-directed with the Choquet integral~\eqref{eq:ChoquetInt},
as in Definition~\ref{def:phi-inf-directed}.

We readily get the following result, as an application of
Theorem~\ref{th:Minimization_IT}.
\begin{proposition}
  Let $\np{\Omega, \mathcal{F}}$ be a measurable space,
  \( \POSET = \left\{ x : \Omega \to \barRR \mid x \geq 0 \right. \)
  and measurable \( \left. \right\} \) be the poset of
  nonnegative measurable functions, and $c$ be a continuous from above capacity .

  If $X = \left\{ x_i \right\}_{i\in I} \subset \POSET$ is a family of nonnegative measurable functions with the
  countable inf property, we have that
  \[
    \winf_{i\in I} \ChoquetInt_{\Omega} x_i \dd c = \ChoquetInt_{\Omega} \winf_{i \in I} x_i \dd c
  \]
  if, and only if, $X$ is Choquet integrably inf-directed.
\end{proposition}

\begin{proof}
  We check that the assumptions of Theorem~\ref{th:Minimization_IT} are
  satisfied.
  \begin{itemize}
  \item
    The set $\POSET = \left\{ x : \Omega \to \barRR \mid x \geq 0 \text{ and measurable} \right\}$
    of nonnegative measurable functions endowed with the pointwise order is an inf-semilattice.
  \item
    The Choquet integral is order preserving on $\POSET$ (see \cite[Proposition 2.3]{Ka2018}).
  \item
    As the capacity $c$ is countinuous from above, the following monotone
    pointwise convergence theorem holds (see \cite[Theorem 3.2.(2)]{Ka2018}): for
    every nonincreasing sequence $\left\{ x_n \right\}_{n\in \NN}$ of functions
    converging pointwise to $x \in \POSET$, we have that
    \[
      \winf_{n\in \NN} \ChoquetInt_{\Omega} x_n \dd c =
      \lim_{n\in \NN}\ChoquetInt_{\Omega} x_n \dd c =  \ChoquetInt_{\Omega} x
      \dd c
      \eqfinp
    \]
    As a consequence, the Choquet integral is sequentially-inf-continuous
    on~$\POSET$         (see Definition~\ref{def:sequentially-inf-continuous}).
  \end{itemize}

  Hence, by Theorem~\ref{th:Minimization_IT},
  given $X = \left\{ x_i \right\}_{i\in I} \subset \POSET$ a family of
  nonnegative functions with the countable inf property, we have
  \[
    \winf_{i\in I} \ChoquetInt_{\Omega} x_i \dd c = \ChoquetInt_{\Omega} \winf_{i \in I} x_i \dd c
  \]
  if, and only if, the subset~$X$ is Choquet integrably inf-directed.
\end{proof}

One could get a similar interchange result between Choquet integral and
maximization on subsets of nonpositive measurable functions by setting, for
every nonpositive measurable function~$x$,
\(  \ChoquetInt_{\Omega} x \dd c = - \ChoquetInt_{\Omega} \np{-x} \dd c \),
where the right hand side is the Choquet integral~\eqref{eq:ChoquetInt} for nonnegative measurable functions.

% \section{Conclusion and perspectives}
%
% \mdl{je propose de supprimer la conclusion}
%
% % We were initially interested in minimization of functions and interchange with mappings $\Phi$ which are not the integral.
% As said in the introduction, the question of interchanging integration and minimization is an important issue
% in stochastic optimization (where integration corresponds to mathematical expectation).
% Moreover, when the mathematical expectation is replaced with a risk measure,
% the question of interchange has been examined in \cite{Pi.Sc2019, Sh2017}.
% An important class of risk measures is made of suprema of integral expressions
% (in case of ``optimistic'' risk measures made of infima of integrals,
% the interchange with minimization would be obvious).
% % {This is why,} in Section~\ref{sec:MIT-main},
% % we started with an abstract result on interchange and minimization followed by an analysis of the integral case.
% There now remains to study under which additional conditions
% our abstract result would apply to suprema of integrals
% and to compare them with the existing literature.
\section{Conclusion}

% We were initially interested in minimization of functions and interchange with mappings $\Phi$ which are not the integral.
As recalled in the introduction, the question of interchanging integration and minimization is an important issue
in optimization (especially in stochastic optimization where integration corresponds to mathematical expectation or to risk measure). By using the framework of posets ---
and especially the notions of semilattice, $\Phi$-inf-directed subset, inf countable subset --- we have provided an umbrella theorem that covers a wide spectrum of results (and extends them). Moreover, our approach goes beyond integration and is able to handle more general monotone functionals.

\appendix

\section{Extended Lebesgue and outer integrals}
\label{sec:outer_int}

The set $\barRR = \RR \cup \na{+\infty} \cup \na{-\infty}$
is endowed with its Borel $\sigma$-algebra~(see %\cite[Application
% 4.2]{Br.Pa2012} or
\cite[Chap. II]{Ne1970}),
and with the following extended additions and multiplication.
We still denote by~$+$ the usual addition when extended to~$\barRR_+ = \RR \cup \na{+\infty} $
by $+\infty$ being absorbant, and to~$\barRR_- = \RR \cup \na{-\infty} $
by $-\infty$ being absorbant.
Then, we denote by~$\LowPlus$ the addition on~$\barRR$ for which
$-\infty$ is absorbant, \emph{i.e.}
$\np{+\infty} \LowPlus (-\infty) = \np{-\infty} \LowPlus \np{+\infty} =-\infty$ and by
$\UppPlus$ the addition for which $+\infty$ is absorbant, \emph{i.e.}
$\np{+\infty} \UppPlus \np{-\infty} = \np{-\infty} \UppPlus \np{+\infty} =
+\infty$.
We set \( \lambda \times \np{\pm\infty}= \pm\infty  \)
for \( \lambda \in ]0,+\infty[ \),
\( \lambda \times \np{\pm\infty}= \mp\infty  \)
for \( \lambda \in ]-\infty,0[ \), and
\( 0 \times \np{\pm\infty}= 0 \).
% \( 0 \times \np{+\infty}= 0 \times \np{-\infty}= 0 \).

Throughout this section, we fix a $\sigma$-finite measured space
$\np{\Omega, \tribu{F}, \mu}$.
The classical Lebesgue integral \wrt\ the $\sigma$-finite measure~$\mu$ is
defined for functions with values in~$\RR$ (real-valued functions).
As we are motivated by optimization, we need results for integrals of functions
with values in~$\barRR$ (extended real-valued functions).
For integration of measurable real-valued functions \wrt\ a $\sigma$-finite measure~$\mu$,
we refer the reader to~\cite[Chapter~11]{Al.Bo2006};
for integration of measurable extended real-valued functions \wrt\ a probability
measure~$\mu$, we refer the reader to~\cite{Ne1970};
for integration of measurable extended real-valued functions \wrt\ a $\sigma$-finite measure~$\mu$,
we refer the reader to~\cite[Chapter~V]{Halmos:1974};
for outer integration of extended real-valued functions \wrt\ a $\sigma$-finite measure~$\mu$,
we refer the reader to~\cite{Be.Sh1996}.

It happens that results about monotonicity, additivity, external multiplication
and monotone convergence of the integral are either scattered in the literature,
or sometimes not formulated.
This is due to the fact that the extension of the Lebesgue integral
to extended real-valued functions gives rise to different expressions,
which renders the exposition less systematic and elegant than with
the Lebesgue integral of integrable real-valued functions.
Also, some results belong to folklore and its is hard to find trace of their
proof, as they are considered obvious.
However, for the purpose of optimizing integral expressions,
% we need to delineate what are general conditions for monotone convergence. This
% is why,
we provide below a systematic exposition of the functional spaces $\Lpmu{0}{\barRR}$, $\Lpmu[\oplus]{1}{\barRR}$
and $\Lpmu[\ominus]{1}{\barRR}$, and how the Lebesgue integral can be extended.

\subsection{Functional space $\Lpmu{0}{\barRR}$ and the Lebesgue integral}

We endow the set~$\barRR^{\Omega}$ of functions $\fonctionun : \Omega \to \barRR$ with the
\emph{$\mu$-pointwise order}~$\leq$ as follows:
for any $\fonctionun ,\fonctiondeux \in \barRR^{\Omega}$,
\begin{equation}
  \fonctionun \leq \fonctiondeux \iff \exists A \in \tribu{F} \eqsepv
  \mu\np{A}=0 \eqsepv \fonctionun \np{\omega} \leq \fonctiondeux\np{\omega}
  \eqsepv \forall \omega \in \Omega \setminus A
  \eqfinp
  \label{eq:pointwise_order}
\end{equation}

We denote by $\mathcal{L}^0\np{\Omega,\tribu{F}; \barRR}$
the set of measurable functions
from $\Omega$ to $\barRR$ and by $\Lpmu{0}{\barRR}$ the quotient
$\mathcal{L}^0\np{ \Omega,\tribu{F};\barRR} /\sim $ where for any
$ \fonctionun, \fonctiondeux  \in \mathcal{L}^0\np{\Omega,\tribu{F}; \barRR}$,
$ \fonctionun \sim \fonctiondeux $ if, and only if, $ \fonctionun = \fonctiondeux $
$\mu$-almost everywhere.
The $\mu$-pointwise order~\eqref{eq:pointwise_order} induces an order on the set~$\Lpmu{0}{\barRR}$
of equivalence classes, that we will also denote by $\leq$ and call the
$\mu$-pointwise order. Thus, the expression \( \fonctionun \geq 0 \)
makes sense for \( \fonctionun \in \Lpmu{0}{\barRR} \).
In the same way, we introduce the \emph{$\mu$-pointwise strict order}~$<$
on the set~$\Lpmu{0}{\barRR}$  of equivalence classes:
\( \fonctionun < \fonctiondeux \iff \fonctionun \leq \fonctiondeux \)
and \( \fonctionun \neq \fonctiondeux \).
Thus, the expressions \( -\infty < \fonctionun \),
\( \fonctionun < +\infty \) and \( -\infty < \fonctionun  < +\infty \)
make sense for \( \fonctionun \in \Lpmu{0}{\barRR} \).

The set~$\Lpmu{0}{\barRR}$ is stable under the two additions~$\LowPlus$ or~$\UppPlus$,
and under external multiplication.
We say that a subset of~$\Lpmu{0}{\barRR}$ is a \emph{convex cone},
if it is stable under the addition~$+$ %two additions~$\LowPlus$ or~$\UppPlus$,
and under external multiplication by a scalar in~$\RR_+ $.

We write $\int$ for the Lebesgue integral deduced from the
$\sigma$-finite measured space $\np{\Omega, \tribu{F}, \mu}$.
The Lebesgue integral~$\int$ is defined on the convex cone
\begin{equation}
  \Lpmu[+]{0}{\barRR} = \bset{ \fonctionun \in \Lpmu{0}{\barRR}}{ \fonctionun \geq 0 }
  \eqfinv
\end{equation}
where it takes values in $\barRR_+$, given by the formula
(see \cite[Footnote~3, p.~411]{Al.Bo2006} for real-valued functions)
\begin{equation}
  \int \fonctionun
  =   \int \fonctionun \dd{\mu}
  = \sup \Bset{ \int_{\Omega} \varphi \dd{\mu}}
  { 0 \le \varphi \le \fonctionun \eqsepv \varphi \text{ simple and
      nonnegative}}
  \eqfinv
  \label{eq:sup-positive-simple}
\end{equation}
where simple nonnegative functions (or $\mu$-step functions) are functions of the form
$\varphi(\cdot) = \sum_{i\in I} \alpha_i\mathbf{1}_{A_i}(\cdot)$ with $I$
finite and $\na{A_i}_{i\in I}$ a sequence of measurable sets such that
$\mu(A_i) < +\infty$ for all $i\in I$ and the coefficients $\na{\alpha_i}_{i\in I}$ are
nonnegative and finite reals and the indicator function $\mathbf{1}_{A}$ of a subset of $\Omega$ is defined by
$\mathbf{1}_{A}(x)=1$ if $x\in A$ and $\mathbf{1}_{A}(x)=0$ if $x\not\in A$.

The (extended) Lebesgue integral on~$\Lpmu[+]{0}{\barRR}$ satisfies the following properties
\begin{itemize}
\item
  monotone:
  $\forall \fonctionun, \fonctiondeux \in \Lpmu[+]{0}{\barRR}$,
  \( \fonctionun \leq \fonctiondeux  \implies \int \fonctionun \leq \int \fonctiondeux  \),
\item
  additive:
  $\forall \fonctionun, \fonctiondeux \in \Lpmu[+]{0}{\barRR}$,
  \( \int \np{ \fonctionun + \fonctiondeux }
  = \int \fonctionun + \int \fonctiondeux  \),
\item
  positively homogeneous:
  $\forall \fonctionun \in \Lpmu[+]{0}{\barRR}$,
  $\forall \lambda \in \RR_+$,
  \( \int \np{\lambda f} = \lambda \int \fonctionun \),
\item
  monotone convergence:
  for any nondecreasing sequence $\np{\fonctionun_n}_{n\in \NN}$
  in~$\Lpmu[+]{0}{\barRR}$, then
  \( \fonctionun = \sup_{n\in \NN} \fonctionun_n \in \Lpmu[+]{0}{\barRR} \)
  and
  \( \lim_{n\to +\infty} \int\fonctionun_n= \int\fonctionun \).
\end{itemize}

\subsection{Functional spaces $\Lpmu[\oplus]{1}{\barRR}$,   $\Lpmu[\ominus]{1}{\barRR}$ and the extended Lebesgue integral}
\label{sec:l1p}
For any function $ \fonctionun : \Omega \to \barRR$,
we define its \emph{positive part}
$ \fonctionun_+ = \sup\np{0,  \fonctionun}$
and its \emph{negative part} $ \fonctionun_- = \sup\np{0, - \fonctionun}$.
Obviously, we have \( \fonctionun = \fonctionun_+ +
\np{ -\fonctionun_- } \) (where we use the addition~$+$ as one of the terms is
zero for any value taken by the argument of the function~$\fonctionun$).
We define the set
\begin{subequations}
  \begin{align}
    \mathcal{L}^1_{\oplus}\np{\Omega,\tribu{F};\barRR}
    & =
      \Bset{  \fonctionun\in \mathcal{L}^0\np{\Omega,\tribu{F}; \barRR}}{\int_{\Omega}  \fonctionun_+
      \dd{\mu} < +\infty}
      \eqfinv
      \intertext{and the quotient set $\mathcal{L}^1_{\oplus}\np{\Omega;
      \barRR}\backslash\sim$ by}
      \label{eq:int_L1rondeplus}
      \Lpmu[\oplus]{1}{\barRR}
    & =
      \Bset{  \fonctionun\in \Lpmu{0}{\barRR}}{\int_{\Omega} \fonctionun_+ \dd{\mu} < +\infty}
      \eqfinv
  \end{align}
  with the property that
  \begin{equation}
    \fonctionun\in \Lpmu[\oplus]{1}{\barRR} \implies
    \fonctionun < +\infty
  \end{equation}
\end{subequations}
because \( \int_{\Omega} \fonctionun_+ \dd{\mu} < +\infty
\implies \fonctionun_+ <  +\infty  \implies
\fonctionun \leq \fonctionun_+ <  +\infty \).
% in the sense that \( \mu\na{\fonctionun = +\infty}=0 \).
%
In the same way, we define
\begin{subequations}
  \begin{align}
    \mathcal{L}^1_{\ominus}\np{\Omega,\tribu{F};\barRR}
    & =
      \Bset{  \fonctionun\in \mathcal{L}^0\np{\Omega,\tribu{F}; \barRR}}{\int_{\Omega}  \fonctionun_-
      \dd{\mu} < +\infty}
      \eqfinv
    \\
    \Lpmu[\ominus]{1}{\barRR}
    & =
      \Bset{  \fonctionun\in \Lpmu{0}{\barRR}}{\int_{\Omega}  \fonctionun_- \dd{\mu} < +\infty}
      \eqfinv
  \end{align}
\end{subequations}
with the properties that
\( \Lpmu[\ominus]{1}{\barRR}= -\Lpmu[\oplus]{1}{\barRR} \) and that
\( \fonctionun\in \Lpmu[\ominus]{1}{\barRR}\)\( \implies
-\infty < \fonctionun \).

% Hence, if $ \fonctionun \in \Lpmu[\oplus]{1}{\barRR}$, then $(-f) \in
% \Lpmu[\ominus]{1}{\barRR}$.

% These four sets are convex cones, as they are stable under the two additions~$\LowPlus$ or~$\UppPlus$,
% and under external multiplication by a scalar
% in~$\barRR_+ = \RR \cup \na{+\infty} $.

We say that a (class of) function(s) $ \fonctionun \in \Lpmu{0}{\barRR} $ is \emph{semi-integrable}
if it belongs to \( \Lpmu[\oplus]{1}{\barRR} \cup
\Lpmu[\ominus]{1}{\barRR} \), that is, if either
\( \int_{\Omega}  \fonctionun_+ \dd{\mu} < +\infty \) or
\( \int_{\Omega}  \fonctionun_- \dd{\mu} < +\infty \).
The Lebesgue integral is extended from the convex cone~$\Lpmu[+]{0}{\barRR}$
to semi-integrable functions by
(\cite[Proposition~II-3-2]{Ne1970}, \cite[Chapter~11]{Al.Bo2006}, \cite[Chapter~V]{Halmos:1974})
% the convex cones~$\Lpmu[\oplus]{1}{\barRR}$ and~$\Lpmu[\ominus]{1}{\barRR}$ by
\begin{equation}
  \int \fonctionun = \int \fonctionun_{+} +
  \Bp{ -\int \fonctionun_{-} } \eqsepv
  \forall \fonctionun \in \Lpmu[\oplus]{1}{\barRR} \cup
  \Lpmu[\ominus]{1}{\barRR}
  \eqfinv
  \label{eq:semi-integrable_Lebesgue_integral}
  % \label{eq:quasi-integrable}
\end{equation}
where we use the addition~$+$ as one of the terms is zero.
The extended Lebesgue integral on semi-integrable functions satisfies the following
properties (listed in~\cite[Proposition II-3-3]{Ne1970})
\begin{itemize}
\item
  monotone:
  $\forall \fonctionun, \fonctiondeux \in \Lpmu[\oplus]{1}{\barRR} \cup
  \Lpmu[\ominus]{1}{\barRR} $,
  \( \fonctionun \leq \fonctiondeux  \implies \int \fonctionun \leq \int \fonctiondeux  \),
\item
  additive on~\( \Lpmu[\oplus]{1}{\barRR} \):
  $\forall \fonctionun, \fonctiondeux \in \Lpmu[\oplus]{1}{\barRR}$,
  \( \int \np{ \fonctionun + \fonctiondeux }
  = \int \fonctionun + \int \fonctiondeux  \),
\item
  additive on~\( \Lpmu[\ominus]{1}{\barRR} \):
  $\forall \fonctionun, \fonctiondeux \in \Lpmu[\ominus]{1}{\barRR}$,
  \( \int \np{ \fonctionun + \fonctiondeux }
  = \int \fonctionun + \int \fonctiondeux  \),
\item
  positively and negatively homogeneous:
  $\forall \fonctionun \in \Lpmu[\oplus]{1}{\barRR} \cup
  \Lpmu[\ominus]{1}{\barRR} $,
  $\forall \lambda \in \RR$,
  \( \int \np{\lambda f} = \lambda \int \fonctionun \),
\item
  monotone convergence on~$\Lpmu[\oplus]{1}{\barRR}$:
  for any nonincreasing sequence $\np{\fonctionun_n}_{n\in \NN}$
  in~$\Lpmu[\oplus]{1}{\barRR}$, then
  \( \fonctionun = \inf_{n\in \NN} \fonctionun_n \in \Lpmu[\oplus]{1}{\barRR} \)
  and
  \( \fonctionun_n \downarrow \fonctionun \) and
  \( \lim_{n\to +\infty} \int\fonctionun_n= \int\fonctionun \),
\item
  monotone convergence on~$\Lpmu[\ominus]{1}{\barRR}$:
  for any nondecreasing sequence $\np{\fonctionun_n}_{n\in \NN}$
  in~$\Lpmu[\ominus]{1}{\barRR}$, then
  \( \fonctionun = \sup_{n\in \NN} \fonctionun_n \in \Lpmu[\ominus]{1}{\barRR} \)
  and
  \( \fonctionun_n \uparrow \fonctionun \) and
  \( \lim_{n\to +\infty} \int\fonctionun_n= \int\fonctionun \).
\end{itemize}

% if $\fonctionun_n \uparrow \fonctionun$ the $\int_{\Omega} \fonctionun_n \dd{\mu} \uparrow \int_{\Omega} \fonctionun \dd{\mu}$ if $\fonctionun_n_{-}$ is integrable for at least one $n\in \NN$.
% if $\fonctionun_n \downarrow \fonctionun$ the $\int_{\Omega} \fonctionun_n \dd{\mu} \downarrow \int_{\Omega} \fonctionun \dd{\mu}$ if $\fonctionun_n_{+}$ is integrable for at least one $n\in \NN$.

We provide some of the proofs.

\begin{lemma}%[Stability of $\stablpmu$ by addition and additivity of the integral]
  \label{lemma:stab_addition_negativepartint}
  For any functions $ \fonctionun$ and $ \fonctiondeux $ in $\Lpmu[\ominus]{1}{\barRR}$, we have
  % Let $ \fonctionun$ and $ \fonctiondeux $ be two elements of
  % $\Lpmu[\ominus]{1}{\barRR}$. Then,
  $ \fonctionun + \fonctiondeux  \in \Lpmu[\ominus]{1}{\barRR}$ and
  \begin{equation*}
    \int_{\Omega} \np{ \fonctionun + \fonctiondeux } \dd{\mu}
    = \int_{\Omega} \fonctionun  \dd{\mu} + \int_{\Omega} \fonctiondeux  \dd{\mu}
    \eqsepv
    \forall \fonctionun \in \Lpmu[\ominus]{1}{\barRR} \eqsepv
    \fonctiondeux  \in \Lpmu[\ominus]{1}{\barRR}
    \eqfinp
  \end{equation*}
\end{lemma}

\begin{proof}
  We consider $ \fonctionun, \fonctiondeux \in \Lpmu[\ominus]{1}{\barRR}$.
  Notice that, as $ \fonctionun, \fonctiondeux \in \Lpmu[\ominus]{1}{\barRR}$,
  we have that \( -\infty < \fonctionun \) and \( -\infty < \fonctiondeux \),
  so that we will use the addition~$+$.

  \noindent $\bullet$
  We show that $\int_{\Omega} \np{ \fonctionun + \fonctiondeux } \dd{\mu} < +\infty$.
  On the one hand, we have
  \[
    \np{ \fonctionun + \fonctiondeux }_-
    = \sup\bp{0, -\np{ \fonctionun + \fonctiondeux }}
    = \sup\bp{0, (- \fonctionun ) + (- \fonctiondeux )}
    \eqfinp
  \]
  On the other hand, we have
  \( (- \fonctionun ) \leq  \fonctionun_{-} \) and
  \( (- \fonctiondeux ) \leq \fonctiondeux_{-} \),
  hence
  $(- \fonctionun ) + (- \fonctiondeux ) \leq  \fonctionun_{-} +
  \fonctiondeux_{-}$
  and thus
  \(  \np{ \fonctionun + \fonctiondeux }_-
  \leq  \fonctionun_{-} + \fonctiondeux_{-}  \).
  By monotonicity and additivity of the Lebesgue integral on
  $\Lpmu[+]{0}{\barRR}$, we deduce that
  \begin{equation*}
    \int_{\Omega} \np{ \fonctionun + \fonctiondeux }_- \dd{\mu}
    \leq \int_{\Omega}  \fonctionun_{-}\dd{\mu} + \int_{\Omega}
    \fonctiondeux_{-}\dd{\mu} < +\infty
    \eqfinv
  \end{equation*}
  because \( \int_{\Omega}  \fonctionun_{-}\dd{\mu}  < +\infty \) and
  \( \int_{\Omega} \fonctiondeux_{-}\dd{\mu} < +\infty \)
  by assumption ($ \fonctionun, \fonctiondeux \in \Lpmu[\ominus]{1}{\barRR}$).
  Hence, $ \fonctionun+ \fonctiondeux  \in \Lpmu[\ominus]{1}{\RR}$.
  \medskip

  \noindent $\bullet$
  We prove the additivity of the integral.
  Notice that, as $ \fonctionun, \fonctiondeux, \fonctionun+\fonctiondeux
  \in \Lpmu[\ominus]{1}{\barRR}$,
  we have that \( -\infty < \fonctionun \) and \( -\infty < \fonctiondeux \),
  and also that
  \( 0 \leq \fonctionun_{-} < +\infty \),
  \( 0 \leq \fonctiondeux_{-} < +\infty \),
  \( 0 \leq \np{\fonctionun+\fonctiondeux}_{-} < +\infty \),
  \( 0 \leq \int_{\Omega} \fonctionun_{-}\dd{\mu} < +\infty \),
  \( 0 \leq \int_{\Omega} \fonctiondeux_{-}\dd{\mu} < +\infty \),
  \( 0 \leq \int_{\Omega} \np{\fonctionun+\fonctiondeux}_{-} < +\infty \),
  so that we will use the addition~$+$.

  As, for any function \( \fonctiontrois \), we have that
  \( \fonctiontrois = \fonctiontrois_{+} + (- \fonctiontrois_{-}) \)
  (where we use the addition~$+$ as one of the terms is zero), we
  immediately get that
  \[
    \np{ \fonctionun  + \fonctiondeux }_+ +
    \bp{-\np{\fonctionun + \fonctiondeux}_-}
    =  \fonctionun + \fonctiondeux
    =  \fonctionun_{+} + (- \fonctionun_{-}) +
    \fonctiondeux_{+} + (- \fonctiondeux_{-})
    \eqfinp
  \]
  Now, if we add, to the left and right hand side of the above equality,
  the three nonnegative reals
  \( \np{\fonctionun + \fonctiondeux}_- \),
  $\fonctionun_{-}$ and $\fonctiondeux_{-}$
  (none of them being $+\infty$), we obtain the equality
  \begin{equation*}
    (f + \fonctiondeux )_+ +  \fonctionun_{-} + \fonctiondeux_{-}
    =   \fonctionun_{+} + \fonctiondeux_{+} + ( \fonctionun  + \fonctiondeux )_-
    \eqfinp
  \end{equation*}
  As this is an equality between sums of nonnegative functions,
  we apply the Lebesgue integral on $\Lpmu[+]{0}{\barRR}$, and get that
  \begin{equation*}
    \int_{\Omega}( \fonctionun  + \fonctiondeux )_+\dd{\mu} + \int_{\Omega}
    \fonctionun_{-}\dd{\mu} + \int_{\Omega} \fonctiondeux_{-}\dd{\mu} =
    \int_{\Omega} \fonctionun_{+}\dd{\mu} + \int_{\Omega}
    \fonctiondeux_{+}\dd{\mu} + \int_{\Omega}( \fonctionun  + \fonctiondeux
    )_-\dd{\mu}
    \eqfinv
  \end{equation*}
  by additivity of the Lebesgue integral on $\Lpmu[+]{0}{\barRR}$.
  Now, the quantities $\int_{\Omega} \fonctionun_{-}\dd{\mu}$,
  $\int_{\Omega} \fonctiondeux_{-}\dd{\mu}$ and $\int_{\Omega}( \fonctionun  +
  \fonctiondeux )_-\dd{\mu}$ are three nonnegative reals
  (none of them being $+\infty$) by assumption
  ($ \fonctionun, \fonctiondeux \in \Lpmu[\ominus]{1}{\barRR}$
  and property $\fonctionun+\fonctiondeux\in \Lpmu[\ominus]{1}{\barRR}$).
  Thus, we get, by subtracting these three finite terms,
  \[
    \begin{split}
      \int_{\Omega}( \fonctionun  + \fonctiondeux )_+\dd{\mu} +
      \bp{- \int_{\Omega}( \fonctionun  + \fonctiondeux )_-\dd{\mu}} \\
      = \int_{\Omega} \fonctionun_{+}\dd{\mu} + \bp{-\int_{\Omega} \fonctionun_{-}\dd{\mu}}
      + \int_{\Omega} \fonctiondeux_+\dd{\mu} + \bp{-\int_{\Omega} \fonctiondeux_{-}\dd{\mu}},
    \end{split}
  \]
  hence, by~\eqref{eq:semi-integrable_Lebesgue_integral},
  \begin{equation*}
    \int_{\Omega}( \fonctionun  + \fonctiondeux ) \dd{\mu} =
    \int_{\Omega} \fonctionun  \dd{\mu} + \int_{\Omega} \fonctiondeux
    \dd{\mu}
    \eqfinp
  \end{equation*}
  This ends the proof.
\end{proof}

\begin{lemma}
  \label{lemma:minus_Lebesgue_integral}
  We have
  \begin{equation*}
    \int_{\Omega} \np{ -\fonctionun } \dd{\mu} = - \int_{\Omega} \fonctionun
    \dd{\mu}
    \eqsepv
    \forall \fonctionun \in \Lpmu[\oplus]{1}{\barRR} \cup
    \Lpmu[\ominus]{1}{\barRR}
    \eqfinp
    \label{eq:minus_Lebesgue_integral}
  \end{equation*}
\end{lemma}

\begin{proof}
  This is an obvious consequence of~\eqref{eq:semi-integrable_Lebesgue_integral},
  and of \( \np{-\fonctionun}_+ = \fonctionun_- \)
  and \( \np{-\fonctionun}_- = \fonctionun_+ \).
  % \begin{equation}
  %   \fonctionun \in \Lpmu[\ominus]{1}{\barRR} \implies
  %   -\fonctionun \in \Lpmu[\oplus]{1}{\barRR} \mtext{ and }
  %   -\int\fonctionun d\mu =\int \np{-\fonctionun} d\mu
  %   \eqfinv
  %   \label{eq:minus_Lebesgue_integral}
  % \end{equation}
\end{proof}

\begin{proposition}[Extended monotone convergence theorem for $\stablpmu$]
  \label{MCT_in_L1moins}
  Let $\np{\fonctionun_n}_{n\in \NN}$ be an nonincreasing sequence of functions
  in~$\Lpmu[\oplus]{1}{\barRR}$, converging to $\fonctionun \in \barRR^{\Omega}$, that
  is, \( \fonctionun_n \downarrow \fonctionun \).
  Then, $\fonctionun \in \Lpmu[\oplus]{1}{\barRR}$ and we have that
  \begin{equation}
    \lim_{n\to +\infty} \int\fonctionun_n d\mu  = \int\fonctionun d\mu
    \eqfinp
    \label{eq:extended_monotone_convergence_theorem}
  \end{equation}
\end{proposition}

\begin{proof}
  Let $\np{\fonctionun_n}_{n\in \NN}$ be an nonincreasing sequence of functions
  in~$\Lpmu[\oplus]{1}{\barRR}$, such that
  \( \fonctionun_n \downarrow \fonctionun \).
  As $ \fonctionun_n \in \Lpmu[\oplus]{1}{\barRR}$,
  we have that \( \fonctionun_n < +\infty \) for all $n \in \NN$, so that we will use the addition~$+$.

  % All functions are in \( \Lpmu{0}{\barRR} \).
  As \( \fonctionun \leq \fonctionun_1 \), we have that
  \( \sup\np{0,\fonctionun} =\fonctionun_{+} \leq \np{\fonctionun_1}_{+}=\sup\np{0,\fonctionun_1} \), hence
  \( \int\fonctionun_{+} d\mu \leq \int\np{\fonctionun_1}_{+} d\mu <+\infty \),
  where the last strict inequality is by assumption
  ($ \fonctionun_1 \in \Lpmu[\oplus]{1}{\barRR}$).
  We conclude that \( \fonctionun \in \Lpmu[\oplus]{1}{\barRR} \).

  As, by assumption, \( \int\np{\fonctionun_1}_{+} d\mu <+\infty \),
  we conclude that \( \np{\fonctionun_1}_{+} <+\infty \).
  % , except on a set of  $\mu$-measure zero.
  We consider two cases.

  We suppose that \( \int\np{\fonctionun_1}_{-} d\mu =+\infty \).
  As \(  \sup\np{0,-\fonctionun} =\fonctionun_{-} \geq
  \np{\fonctionun_1}_{-}=\sup\np{0,-\fonctionun_1} \), we also have that
  \( \int\fonctionun_{-} d\mu =+\infty \).
  As a consequence, we get that
  \( \int\fonctionun_{-} d\mu = \int\np{\fonctionun_1}_{-} d\mu =+\infty \),
  hence \( \int\fonctionun d\mu = \int\fonctionun_1 d\mu =+\infty \),
  by definition of the integral~$\int$ on~\( \Lpmu[\oplus]{1}{\barRR} \).
  By monotonicity of the integral~$\int$, we conclude that
  \( +\infty=\int\fonctionun_1 d\mu \leq
  \lim_{n\to +\infty} \int\fonctionun_n d\mu \leq
  \int\fonctionun_1 d\mu =+\infty \),
  hence that~\eqref{eq:extended_monotone_convergence_theorem} holds true.

  We now suppose that \( \int\np{\fonctionun_1}_{-} d\mu <+\infty \).
  We deduce that \( \np{\fonctionun_1}_{-} <+\infty \).
  % , except on a set of  $\mu$-measure zero.
  As we had \( \np{\fonctionun_1}_{+} <+\infty \),
  % , except on a set of  $\mu$-measure zero,
  we deduce that \( -\infty < \fonctionun_1 <+\infty \).
  % , except on the union  of two sets of $\mu$-measure zero.
  % From now on, we ignore this latter set, which is of $\mu$-measure zero.
  Thus, we can define
  \( \varphi_n =  \fonctionun_n + \np{-\fonctionun_1} \)
  and \( \varphi= \fonctionun + \np{-\fonctionun_1} \),
  which are functions in \( \Lpmu{0}{\barRR} \) such that
  \(
  \varphi= \fonctionun + \np{-\fonctionun_1}
  % = \fonctionun - \fonctionun_1
  \leq
  \varphi_n =  \fonctionun_n + \np{-\fonctionun_1}
  \leq
  0 \),
  % = \fonctionun_n - \fonctionun_1
  because \( \fonctionun_n \leq \fonctionun_1 \). %\( \fonctionun_1 \in \RR \).
  As \( \fonctionun_1 \) takes values in~\( \RR \), we have that
  \( \inf_{n} \np{\fonctionun_n +  \np{-\fonctionun_1} }
  =  \inf_{n} \fonctionun_n + \np{-\fonctionun_1} \),
  % by property of the upper Moreau addition,
  hence we obtain that
  \( \varphi_n \downarrow \varphi\). As \( \varphi_n \leq 0 \),
  by the monotone convergence theorem for ($\Lpmu[-]{0}{\barRR}$, $\int$),
  we get that
  \[
    \inf_{n} \int \varphi_n d\mu % \np{\fonctionun_n- \fonctionun_1} d\mu
    = \lim_{n\to +\infty} \int \varphi_n d\mu % \np{\fonctionun_n- \fonctionun_1} d\mu
    = \int \varphi d\mu % \np{\fonctionun - \fonctionun_1} d\mu
    \eqfinp
  \]
  As, by assumption, \( \int \np{\fonctionun_1}_{-} d\mu <+\infty \) and
  \( \int\np{\fonctionun_1}_{+} d\mu <+\infty \), we get that
  \( \fonctionun_1 \in  \Lpmu{1}{\barRR} \) and that
  \(  -\infty < \int\fonctionun_1 d\mu <+\infty \), hence obtaining
  % By property of the upper Moreau addition, we get that
  \[
    % \lim_{n\to +\infty} \bp{ \int \varphi_n d\mu % \np{\fonctionun_n- \fonctionun_1} d\mu
    \inf_{n} \bp{ \int \varphi_n d\mu % \np{\fonctionun_n- \fonctionun_1} d\mu
      + \int\fonctionun_1 d\mu }
    =
    \inf_{n} \int \varphi_n d\mu % \np{\fonctionun_n- \fonctionun_1} d\mu
    + \int\fonctionun_1 d\mu
    =
    \int \varphi d\mu % \np{\fonctionun - \fonctionun_1} d\mu
    + \int\fonctionun_1 d\mu
    \eqfinp
  \]
  As \( \varphi_n \leq 0 \) and belongs to
  \( \Lpmu{0}{\barRR} \), we have that \( \varphi_n \in \Lpmu[\oplus]{1}{\barRR}
  \). In the same way, we obtain that \( \varphi \in \Lpmu[\oplus]{1}{\barRR}\).
  By the $+$-additivity property of the integral~$\int$
  on \( \Lpmu[\oplus]{1}{\barRR}\), we calculate the first and last terms of the
  above equality, and we obtain
  \[
    \inf_{n} \int \np{ \varphi_n + \fonctionun_1 } d\mu
    = \int \np{ \varphi + \fonctionun_1} d\mu \eqfinp
  \]
  We obtain~\eqref{eq:extended_monotone_convergence_theorem}
  because \( \varphi_n + \fonctionun_1=
  \fonctionun_n + \np{-\fonctionun_1} + \fonctionun_1=
  \fonctionun_n \) since \( \fonctionun_1 \) takes values in~\( \RR \),
  % \( \fonctionun_1 \in \RR \),
  and, in the same way,
  \( \varphi + \fonctionun_1=
  \fonctionun + \np{-\fonctionun_1} + \fonctionun_1
  =\fonctionun\).

\end{proof}

The classical \emph{vector space of integrable functions} is
\begin{equation}
  \Lpmu{1}{\barRR} =
  \Lpmu[\oplus]{1}{\barRR} \cap
  \Lpmu[\ominus]{1}{\barRR}
  \eqfinv
\end{equation}
with the property that \( \fonctionun\in \Lpmu{1}{\barRR}\)\( \implies
-\infty < \fonctionun < +\infty \), that is,
\( \Lpmu{1}{\barRR} = \Lpmu{1}{\RR} \).

\subsection{Outer integral on $\Lpmu{0}{\barRR}$}

We follow~\cite{Be.Sh1996} for the following definitions.

\begin{definition}
  \label{def:Le-outer-L0}
  We define the \emph{outer integral} of a function by
  \begin{subequations}
    \begin{align}
      \int^{*}_{\Omega} \fonctionun \dd{\mu}
      & = \inf \Bset{ \int_{\Omega} \psi \dd{\mu}}
        {\psi \in \Lpmu{1}{\RR} \text{ and }  \fonctionun \le \psi }
        \eqsepv \forall \fonctionun \in \barRR^{\Omega}
        \eqfinv
        \label{eq:Le-outer-L0}
        \intertext{and the \emph{inner integral} by}
        \int_{*}^{\Omega} \fonctionun \dd{\mu}
      & = \sup \Bset{ \int_{\Omega} \psi \dd{\mu}}
        {\psi \in \Lpmu{1}{\RR} \text{ and }  \fonctionun \ge \psi }
        \eqsepv \forall \fonctionun \in \barRR^{\Omega}
        \eqfinv
    \end{align}
  \end{subequations}
  where $\int_{\Omega} \psi \dd{\mu}$ is the classical Lebesgue integral for $\psi \in \Lpmu{1}{\RR}$.
\end{definition}
It is straightforward that
\begin{subequations}
  \begin{align}
    \int_{*}^{\Omega} \fonctionun \dd{\mu}
    & \leq \int^{*}_{\Omega} \fonctionun \dd{\mu}
      \eqsepv \forall \fonctionun \in \barRR^{\Omega}
      \eqfinv
    \\
    - \int^{*}_{\Omega} \fonctionun \dd{\mu}
    & \leq
      \int^{*}_{\Omega} \np{-\fonctionun} d\mu
      \eqsepv
      \forall \fonctionun \in \barRR^{\Omega}
      \eqfinv
    \\
    \int_{*}^{\Omega} \fonctionun \dd{\mu}
    & =
      - \bp{ \int^{*}_{\Omega} \np{-\fonctionun} \dd{\mu} }
      \eqsepv \forall \fonctionun \in \barRR^{\Omega}
      \eqfinp
      \label{eq:Le-outer=minusminusLe-inner}
  \end{align}
\end{subequations}

% We introduce the outer integral on the set~$\Lpmu{0}{\barRR}$ as follows.
% \begin{definition}[outer integral on $\Lpmu{0}{\barRR}$]
%   \label{def:Le-outer-L0}
%   The outer integral of a function $\fonctionun \in \Lpmu{0}{\barRR}$ is defined by
%   \begin{align}
%     \int^{*}_{\Omega} \fonctionun \dd{\mu}
%     & = \inf \Bset{ \int_{\Omega} \psi \dd{\mu}}
%     {\psi \in \Lpmu{1}{\RR} \text{ and }  \fonctionun \le \psi \text{ \Pps[\mu]} }
%     \eqfinv
%     \label{eq:Le-outer-L0}
%   \end{align}
%   where $\int_{\Omega} \psi \dd{\mu}$ is the classical Lebesgue integral for $\psi \in \Lpmu{1}{\RR}$.
% \end{definition}

These outer and inner integrals extend
% This outer integral extends
the classical Lebesgue integral to the uncovered case where both
$\int_{\Omega} \fonctionun_{+} \dd{\mu}$ and $\int_{\Omega} \fonctionun_{-} \dd{\mu}$ equal $+\infty$ as shown in the
following Proposition.

\begin{proposition}
  \label{outerIsExtendedL}
  We have that
  \begin{subequations}
    \begin{align}
      \int^{*}_{\Omega} \fonctionun \dd{\mu}
      & =
        \int_{\Omega} \fonctionun_{+} \dd{\mu}
        \UppPlus\Bp{ - \int_{\Omega} \fonctionun_{-} \dd{\mu}}
        \eqsepv
        \forall \fonctionun \in \Lpmu{0}{\barRR}
        \eqfinv
        \label{eq:outer-as-leb-diff}
      \\
      \int_{*}^{\Omega} \fonctionun \dd{\mu}
      & =
        \int_{\Omega} \fonctionun_{+} \dd{\mu}
        \LowPlus\Bp{ - \int_{\Omega} \fonctionun_{-} \dd{\mu}}
        \eqsepv
        \forall \fonctionun \in \Lpmu{0}{\barRR}
        \eqfinp
        \label{eq:inner-as-leb-diff}
    \end{align}
  \end{subequations}
  % \begin{equation}
  %   \int^{*}_{\Omega} \fonctionun \dd{\mu} =
  %   \int_{\Omega} \fonctionun_{+} \dd{\mu}
  %   \UppPlus\Bp{ - \int_{\Omega} \fonctionun_{-} \dd{\mu}}
  %   \eqsepv
  %   \forall \fonctionun \in \Lpmu{0}{\barRR}
  %   \eqfinv
  %   \label{eq:outer-as-leb-diff}
  % \end{equation}
  % where the outer integral of $\fonctionun$ is defined in Equation~\eqref{eq:Le-outer-L0}.
  As a consequence, the outer integral of $\fonctionun$ coincides with the
  extended Lebesgue integral~\eqref{eq:semi-integrable_Lebesgue_integral}  on \( \Lpmu[\oplus]{1}{\barRR} \cup
  \Lpmu[\ominus]{1}{\barRR} \), that is, when $\fonctionun$ is semi-integrable.
\end{proposition}

\begin{proof}
  % \noindent $\bullet$
  % Second, the Lebesgue integral for
  % $\fonctionun \in \Lpmu{0}{\barRR}$ is defined by
  % \begin{align}
  %   \int \fonctionun \dd{\mu}
  %   & =   \int \fonctionun_{+} \dd{\mu} - \int \fonctionun_{-} \dd{\mu}\eqsepv
  %   \label{eq:quasi-integrable}
  % \end{align}
  % for a subset of $\Lpmu{0}{\barRR}$, named the \emph{semi-integrable} functions, which
  % consists of function for which the right hand side of
  % Equation~\eqref{eq:quasi-integrable} is properly defined. That is when the two
  % Lebesgue integrals $\int \fonctionun_{+} \dd{\mu}$ and $\int \fonctionun_{-} \dd{\mu}$ are
  % not \emph{both} equal to $+\infty$. When both integrals are finite the
  % function $\fonctionun$ is said to be integrable, the set of such functions being denoted by $\Lpmu{1}{\RR}$.

  % \noindent $\bullet$ Now, assuming that
  % \noindent $\bullet$
  We consider $\fonctionun \in \Lpmu{0}{\barRR}$ and
  % semi-integrable, and we prove that
  % \begin{equation}
  %   \label{lebesgue-rev}
  %   \int_{\Omega} \fonctionun \dd{\mu} =
  %   \inf \Bset{ \int_{\Omega} \psi \dd{\mu}}{\psi \in \Lpmu{1}{\RR} \text{ and }  \fonctionun \le \psi \text{ \Pps[\mu]} }\eqfinp
  % \end{equation}
  % % Thus, for semi-integrable functions the Lebesgue integral defined by Equations~\eqref{eq:sup-positive-simple}
  % % and~\eqref{eq:quasi-integrable} coincide with the outer integral defined by Equation~\eqref{eq:Le-outer-L0}.
  % % We prove this fact considering the three possible cases.
  % We
  we  examine four possible cases in order to
  prove Equation~\eqref{eq:outer-as-leb-diff}
  (then Equation~\eqref{eq:inner-as-leb-diff} is obtained from~\eqref{eq:Le-outer=minusminusLe-inner}).
  \medskip

  \noindent $\bullet$
  Suppose that $\int_{\Omega} \fonctionun_{+} \dd{\mu} < +\infty$
  and $\int_{\Omega} \fonctionun_{-} \dd{\mu} < +\infty$
  (that is, $\fonctionun \in \Lpmu{1}{\RR}$). Then
  % by~\cite[Proposition   II-3-3]{Ne1970}]
  we have that $\mu\nc{\na{\fonctionun = \pm \infty}}=0$, and thus
  there exists a representant $\tilde{\fonctionun} \in \Lpmu{1}{\RR}$ in the
  class, which is equal to $\fonctionun$  (\Pps[\mu]).  Thus, we have that
  $\int^{*}_{\Omega} \fonctionun \dd{\mu} \le \int_{\Omega} \tilde{\fonctionun} \dd{\mu} =
  \int_{\Omega} \fonctionun \dd{\mu}$ as we can use $\psi = \tilde{\fonctionun}$ in the
  definition of the outer integral.  Now, in order to prove the reverse inequality $\int_{\Omega} \fonctionun \dd{\mu} \leq \int^{*}_{\Omega} \fonctionun \dd{\mu}$, we have to consider two cases, depending whether
  $\int^{*}_{\Omega} \fonctionun \dd{\mu}$ is finite or is equal to $-\infty$.
  \smallskip

  \noindent $\diamond$
  In the case where        $\int^{*}_{\Omega} \fonctionun \dd{\mu}$ is finite, we fix $\epsilon >0$. Using
  Equation~\eqref{eq:Le-outer-L0}, there exists $\psi_{\epsilon} \in \Lpmu{1}{\RR}$ such
  that $\fonctionun \le \psi_{\epsilon}$ and
  $\int_{\Omega} \psi_{\epsilon} \dd{\mu} \le \int^{*}_{\Omega} \fonctionun
  \dd{\mu} + \epsilon$. Using the fact that $\fonctionun \in
  \Lpmu{1}{\RR}$ %is integrable
  and the monotonicity of the Lebesgue integral, we obtain
  \[
    \int_{\Omega} \fonctionun \dd{\mu} \le \int_{\Omega}
    \psi_{\epsilon}\dd{\mu}
    \le \int^{*}_{\Omega} \fonctionun \dd{\mu} + \epsilon
    \eqfinv
  \]
  which finally gives
  $\int_{\Omega} \fonctionun \dd{\mu} \le \int^{*}_{\Omega} \fonctionun \dd{\mu}$ and therefore the equality
  $\int_{\Omega} \fonctionun \dd{\mu}= \int^{*}_{\Omega} \fonctionun \dd{\mu}$.
  Equation~\eqref{eq:outer-as-leb-diff} follows using
  Equation~\eqref{eq:semi-integrable_Lebesgue_integral} as we have
  \[
    \int^{*}_{\Omega} \fonctionun \dd{\mu} =
    \int_{\Omega} \fonctionun \dd{\mu} = \int_{\Omega} \fonctionun_+
    \dd{\mu}
    + \np{ - \int_{\Omega} \fonctionun_- \dd{\mu} }
    = \int_{\Omega} \fonctionun_+ \dd{\mu} \UppPlus \bp{- \int_{\Omega} \fonctionun_-  \dd{\mu}}
    \eqfinp
  \]
  \smallskip

  \noindent $\diamond$
  In the case where   $\int^{*}_{\Omega} \fonctionun \dd{\mu}= -\infty$, then using Equation~\eqref{eq:Le-outer-L0}
  there exists a sequence $\ba{\psi_{n}}_{n\in \NN}$
  in $\Lpmu{1}{\RR}$ such that $\fonctionun \le \psi_n$ and $\int_{\Omega} \psi_{n}\dd{\mu} \le -n$ for all
  $n\in \NN$. This implies that $\int_{\Omega} \fonctionun \dd{\mu} =
  -\infty$, which contradicts the fact that $\fonctionun \in \Lpmu{1}{\RR}$.
  \medskip

  \noindent $\bullet$
  Suppose that $\int_{\Omega} \fonctionun_{+} \dd{\mu} < +\infty$
  and $\int_{\Omega} \fonctionun_{-} \dd{\mu} = +\infty$.
  Using the fact that $\fonctionun \le \fonctionun_{+}$, we get that
  $\int^{*}_{\Omega} \fonctionun \dd{\mu} \le \int_{\Omega} \fonctionun_{+} \dd{\mu}$
  as we can use $\psi = \fonctionun_{+} \in \Lpmu{1}{\RR}$ in the
  definition~\eqref{eq:Le-outer-L0} of the outer integral.
  Moreover, as $\int_{\Omega} \fonctionun_{-} \dd{\mu} = +\infty$,
  we can find a sequence  $\na{\psi_n}_{n\in \NN}$ of nonnegative functions such that $\psi_n \in \Lpmu{1}{\RR}$,
  $\psi_n \le \fonctionun_{-}$ and such that $\lim_{n\to \infty} \int_{\Omega} \psi_{n} \dd{\mu} = +\infty$
  for all $n\in \NN$ (take $\psi_n = \1_{\Omega_n} \min\np{n,\fonctionun_{-}}$, where $\np{\Omega_n}_{n \in \NN}$ is a
  monotone sequence of $\tribu{F}$-measurable subsets of~$\Omega$ covering $\Omega$ such that $\mu(\Omega_n) < +\infty$
  which exists by $\sigma$-finite property). Using the fact that $\int_{\Omega} \fonctionun_{+} \dd{\mu} < +\infty$,
  we can find $\tilde{\fonctionun} \in \Lpmu{1}{\RR}$ such that $\fonctionun_{+}= \tilde{\fonctionun}$ \Pps[\mu]
  Thus, for all $n\in \NN$, we have that $\fonctionun \le \np{\tilde{\fonctionun} - \psi_n}$ and
  $\np{\tilde{\fonctionun} - \psi_n} \in \Lpmu{1}{\RR}$. We obtain,
  using monotonicity and monotone convergence that
  \[
    \int^{*}_{\Omega} \fonctionun \dd{\mu} \le
    \int_{\Omega} \np{ \tilde{\fonctionun} - \psi_n } \dd{\mu}
    = \int_{\Omega} \fonctionun_{+}\dd{\mu} - \int_{\Omega} \psi_n \dd{\mu}
    \mathop{\to}_{n \to +\infty} - \infty
    \eqfinp
  \]
  We therefore obtain~Equation~\eqref{eq:outer-as-leb-diff} since both members of the equality are equal to
  $-\infty$.

  % \noindent $\diamond$ (second version) Assume that $\int_\Omega \fonctionun_+\dd{\mu} \in \RR$ and
  % $\int_\Omega \fonctionun_{-}\dd{\mu} = + \infty$. Let $\rho \in \Lpmu[\oplus]{1}{\RR}$ be a
  % given positive integrable function and consider the sequence
  % $\na{\psi_n}_{n\in \NN}$ defined by $\psi_n = \fonctionun_{+} - \inf \np{\fonctionun_{-}, n\rho}$ for all $n\in \NN$.
  % Then, for all $n\in \NN$, we have that
  % $\psi_n \in \Lpmu{1}{\RR}$ and $\fonctionun \le \psi_n$. Thus, we obtain that
  % \begin{align*}
  %   \inf & \Bset{ \int_{\Omega} \psi \dd{\mu}}{\psi \in \Lpmu{1}{\RR} \text{ and }  \fonctionun \le \psi \text{ \Pps[\mu]} }
  %   \\
  %   &\hspace{1cm}\le \inf_{n\in \NN} \int_{\Omega} \psi_n \dd{\mu}
  %   \tag{as $\nset{\psi_n}{n\in \NN}\subset \nset{\psi \in \Lpmu{1}{\RR}}{ \fonctionun\le \psi}$}
  %   \\
  %   &\hspace{1cm}= \int_{\Omega} \fonctionun_{+} \dd{\mu} - \sup_{n\in \NN} \int_{\Omega}\inf(\fonctionun_{-}, n\rho) \dd{\mu}
  %   \\
  %   &\hspace{1cm}= \int_{\Omega} \fonctionun_{+} \dd{\mu} - \int_{\Omega} \sup_{n\in \NN} \inf(\fonctionun_{-}, n\rho) \dd{\mu}
  %   \tag{monotone convergence in \cite[Proposition II-3-3]{Ne1970}]}
  %   \\
  %   &\hspace{1cm}=  \int_{\Omega} \fonctionun_{+}\dd{\mu} - \int_{\Omega} \fonctionun_{-} \dd{\mu} = -\infty
  %   \eqfinp
  % \end{align*}
  % We therefore obtain Equation~\eqref{lebesgue-rev}.
  \medskip

  \noindent $\bullet$
  Suppose that $\int_{\Omega} \fonctionun_{+} \dd{\mu} = +\infty$
  and $\int_{\Omega} \fonctionun_{-} \dd{\mu} < +\infty$. Then we prove that
  $$\nset{\psi \in \Lpmu{1}{\RR}}{\fonctionun \le \psi \text{
      \Pps[\mu]}}=\emptyset \eqfinp$$ Indeed, assuming the existence of
  $\psi \in \Lpmu{1}{\RR}$ such that $\fonctionun \le \psi$, we would obtain that
  $\fonctionun_+ \le \psi + \fonctionun_{-}$ which, using the fact that
  $\psi + \fonctionun_{-} \in \Lpmu{1}{\RR}$, would imply that
  $\int_{\Omega} \fonctionun_{+} \dd{\mu} < +\infty$, hence
  % $\fonctionun_+ \in   \Lpmu{1}{\RR}$,
  contradicting the assumption that
  $\int_{\Omega} \fonctionun_{+} \dd{\mu} = +\infty$.
  \medskip

  \noindent $\bullet$
  Suppose that $\int_{\Omega} \fonctionun_{+} \dd{\mu} = +\infty$
  and $\int_{\Omega} \fonctionun_{-} \dd{\mu} = +\infty$. Using the definition
  of~ $\UppPlus$, we get that the right hand side of Equation~\eqref{eq:outer-as-leb-diff} is equal to $+\infty$.
  Now, we show that Equation~\eqref{eq:outer-as-leb-diff} holds true by proving that
  the set of functions $\psi \in \Lpmu{1}{\RR}$ such that $\fonctionun \le \psi$
  is empty. We proceed by contradiction.
  Assuming the existence of $\psi \in \Lpmu{1}{\RR}$ such that $\fonctionun \le
  \psi$, we would have
  \[
    +\infty=  \int_{\Omega} \fonctionun_{+} \dd{\mu}
    = \int_{\Omega} \fonctionun\1_{\fonctionun\ge 0} \dd{\mu}
    \le \int_{\Omega} \psi\1_{\fonctionun\ge 0} \dd{\mu}
    \le  \int_{\Omega}\psi \dd{\mu}
    \eqfinv
  \]
  contradicting the assumption that $\psi \in \Lpmu{1}{\RR}$.
  Therefore, in Equation~\eqref{eq:Le-outer-L0}
  we obtain that
  $\int^{*}_{\Omega} \fonctionun \dd{\mu} = +\infty$
  and thus equality is ensured in Equation~\eqref{eq:outer-as-leb-diff}.
  \medskip

  This ends the proof.
\end{proof}

\bibliographystyle{abbrv} % smfalpha}
\bibliography{biblio_InversionInfEsperance}

\end{document}